  \documentclass[11pt,letterpaper,twoside]{amsart}

  \usepackage[margin=4cm]{geometry}

  \usepackage{amsmath,amssymb,amsthm,amsfonts,bbding,stmaryrd,dsfont,wasysym,pifont,xspace,xcolor}
  \usepackage{parskip,fancyhdr,url}

  \usepackage{epstopdf,yfonts}
  \usepackage{graphicx,enumerate,epsfig} 
  \usepackage{mathrsfs}
  \usepackage{enumitem}
  \usepackage[normalem]{ulem}

  \usepackage[
     breaklinks,colorlinks,
     citecolor=blue,linkcolor=blue,urlcolor=teal
  ]{hyperref}

  \usepackage{tikz}
  \usetikzlibrary{arrows,calc,decorations.pathmorphing,backgrounds,positioning,fit}

  \usepackage[all]{xy}

  \begingroup
    \makeatletter
    \@for\theoremstyle:=definition,remark,plain\do{%
      \expandafter\g@addto@macro\csname th@\theoremstyle\endcsname{%
        \addtolength\thm@preskip\parskip
        }%
      }
  \endgroup

  \newtheorem{theorem}{Theorem}[section]
  \newtheorem{proposition}[theorem]{Proposition}
  \newtheorem{corollary}[theorem]{Corollary}
  \newtheorem{lemma}[theorem]{Lemma}

  \theoremstyle{definition}
  \newtheorem{definition}[theorem]{Definition}
  \newtheorem{remark}[theorem]{Remark}
  \newtheorem{claim}[theorem]{Claim}
  \newtheorem*{claim*}{Claim}

  \newtheorem*{question*}{Question}
  \newtheorem*{answer*}{Answer}
  \newtheorem*{application*}{Application}

  \newcommand{\secref}[1]{Section~\ref{Sec:#1}}
  
  \newcommand{\thmref}[1]{Theorem~\ref{Thm:#1}}
  \newcommand{\corref}[1]{Corollary~\ref{Cor:#1}}
  \newcommand{\lemref}[1]{Lemma~\ref{Lem:#1}}
  \newcommand{\propref}[1]{Proposition~\ref{Prop:#1}}

  \newcommand{\defref}[1]{Definition~\ref{Def:#1}}
  
  \newcommand{\eqnref}[1]{Equation~\ref{Eqn:#1}}

  \newcommand{\calC}{\mathcal{C}}

  \newcommand{\calL}{\mathcal{L}}

  \newcommand{\calP}{\mathcal{P}}
  \newcommand{\calQ}{\mathcal{Q}}

  \newcommand{\calT}{\mathcal{T}}

  \newcommand{\calX}{\mathcal{X}}
  \newcommand{\calY}{\mathcal{Y}}

   \newcommand{\g}{{\bf g}}

  \DeclareMathOperator{\diam}{diam}

  \DeclareMathOperator{\I}{i}

  \DeclareMathOperator{\Log}{Log}

   \DeclareMathOperator{\len}{length}
  \DeclareMathOperator{\ssm}{\smallsetminus}

  \newcommand{\emul}{\stackrel{{}_\ast}{\asymp}}
  \newcommand{\gmul}{\stackrel{{}_\ast}{\succ}}
  \newcommand{\lmul}{\stackrel{{}_\ast}{\prec}}
  \newcommand{\eadd}{\stackrel{{}_+}{\asymp}}
  
  \newcommand{\ladd}{\stackrel{{}_+}{\prec}}
  
   \newcommand{\bd}{\partial}

  \newcommand{\HH}{\ensuremath{\mathbf{H}}\xspace}
  
  \newcommand{\RR}{\ensuremath{\mathbf{R}}\xspace}

  \newcommand{\s}{\ensuremath{S}\xspace} 
   
  \newcommand{\T}{\ensuremath{\calT}\xspace} 
   
  \newcommand{\Lam}{\ensuremath{\mathcal{GL}}\xspace} 
   
  \newcommand{\set}[1]{\ensuremath{\left\{ {#1} \right\}}\xspace} 
   
  \newcommand{\st}{\ensuremath{\,\,\, \colon \,\,\,}\xspace}

  \newcommand{\Teich}{{Teichm\"uller }} 

  \newcommand{\ep}{\ensuremath{\epsilon}\xspace} 

  \newcommand{\dth}{\ensuremath{d_{\text{Th}}}\xspace} 
    
  \newcommand{\dH}{\ensuremath{d_{\HH}}\xspace}

  \newcommand{\talpha}{{\widetilde\alpha}}
  \newcommand{\balpha}{{\overline\alpha}}
  
  \newcommand{\tbeta}{{\widetilde\beta}}
  \newcommand{\tlambda}{{\widetilde\lambda}}
  \newcommand{\tomega}{{\widetilde\omega}}
  \newcommand{\tgamma}{{\widetilde\gamma}}
  \newcommand{\bgamma}{{\overline\gamma}}
  \newcommand{\blambda}{{\overline\lambda}}
  \newcommand{\bomega}{{\overline\omega}}
  
  \newcommand{\teta}{{\widetilde\eta}}
  \newcommand{\tf}{{\widetilde f}}
  \newcommand{\tl}{{\widetilde l}}
  \newcommand{\tp}{{\widetilde p}}

  \newcommand{\tX}{{\widetilde X}}

  \newcommand{\param}{{\mathchoice{\mkern1mu\mbox{\raise2.2pt\hbox{$
  \centerdot$}}
  \mkern1mu}{\mkern1mu\mbox{\raise2.2pt\hbox{$\centerdot$}}\mkern1mu}{
  \mkern1.5mu\centerdot\mkern1.5mu}{\mkern1.5mu\centerdot\mkern1.5mu}}}

  \newcommand{\from}{\colon\,}

  \fancypagestyle{plain}
  
  \numberwithin{equation}{section}

\begin{document}


  \title{Thurston geodesics: no backtracking and active intervals}
  \author   {Anna Lenzhen} 
  \address  {Department of Mathematics\\
             University of Rennes\\
             Rennes, FRANCE} 
  \email    {anna.lenzhen@univ-rennes1.fr}
  \author   {Babak Modami} 
  \address  {School of Mathematics\\
            Tata Institute of Fundamental Research \\
             Mumbai, IN}
  \email    {bmodami@math.tifr.res.in}
  \author   {Kasra Rafi}
  \address  {Department of Mathematics\\
             University of Toronto\\
             Toronto, CANADA} 
  \email    {rafi@math.toronto.edu}
  \author   {Jing Tao}
  \address  {Department of Mathematics\\
             University of Oklahoma\\
             Norman, OK 73019-0315, USA}
  \email    {jing@ou.edu}

  \maketitle
  \thispagestyle{empty}
 
  \begin{abstract}

     We develop the notion of the {\em active interval} for a subsurface
     along a geodesic in the Thurston metric on \Teich space of a surface
     $S$. That is, for any geodesic in the Thurston metric and any
     subsurface $R$ of $S$, we find an interval of times where the length
     of the boundary of $R$ is uniformly bounded and the restriction of the
     geodesic to the subsurface $R$ resembles a geodesic in the \Teich
     space of $R$. In particular, the set of short curves in $R$ during the
     active interval represents a reparametrized quasi-geodesic in the
     curve graph of $R$ (no backtracking) and the amount of movement in the
     curve graph of $R$ outside of the active interval is uniformly bounded
     which justifies the name \emph{active interval}. These intervals
     provide an analogue of the active intervals introduced by the third
     author in the setting of \Teich space equipped with the \Teich metric. 
    
  \end{abstract}
  
  \setcounter{tocdepth}{2}


\section{Introduction}

   This is the third paper in the series \cite{LRT1,LRT2} where the Thurston
   metric on \Teich space and the  behavior of its geodesics is studied. 
   Let \s be a surface of finite hyperbolic type and let $\T(S)$ be the
   \Teich space of \s equipped with the Thurston metric $\dth$ (see
   \secref{ThurstonMetric} for definition). It is well known that there may
   not be a unique Thurston geodesic connecting two points $X,Y \in \T(\s)$
   (see \cite{thurston-stretch} or \cite[Theorem 1.1]{LRT1}). In general,
   for a subsurface $R \subseteq S$, the projection of a geodesic segment
   $[X,Y]$ to the \Teich space of $R$ can be essentially any path in
   $\T(R)$. Heuristically, this can be understood as follows:  
   let $X,Y$ be points in $\T(\s)$ and let $R$, $R'$ be disjoint subsurfaces
   such that the length of $\bd R$ is short in both $X$ and $Y$, the
   restriction of the hyperbolic structures of $X$ and $Y$ to $R$ is the same 
   but the restriction of $X$ and $Y$ to $R'$ is very different. Then any geodesic
   segment from $X$ to $Y$ has to move efficiently in $R'$ but it has the
   freedom to modify the hyperbolic structure inside $R$ in different ways
   before returning back to the original structure (see \cite[Section 6]{LRT2}
   for a detailed examples of such phenomena). In other words, a copy 
   of $\calT(R) \times \calT(R')$ equipped with the $L^\infty$--metric embeds 
   nearly isometrically in $\calT(S)$ (see \cite[Theorem 3.5]{compteichlip})
   which causes non-uniqueness of geodesics similar to that of $\RR^2$
   equipped with the $L^\infty$--metric. 
   
   In this paper we will show that, in a coarse sense, this is essentially
   the only phenomenon responsible for the non uniqueness of geodesics.  We
   now make this precise.
   
  For a point $X \in \T(S)$, let $\mu_X$ be the \emph{short marking} of $X$
  (see \secref{marking} for the definition).
  For a subsurface $R\subseteq S$, let $\calC(R)$ be the curve graph of 
  the subsurface $R$. Now define   
   \[\Upsilon_R: \T(S) \to \calC(R)\] to be the map which
   assigns to a point $X$ the projection of $\mu_X$ on $X$ to $\calC(R)$ (see \defref{pi}). 
   We call $\Upsilon_R$ the shadow of $X$ to the curve complex of $R$.

   Let $\g \from I \to \T(S)$ be a geodesic segment from $X$ to $Y$ and let
   $\lambda_\g$ be the maximally stretched lamination from $X$ to $Y$
   introduced by Thurston \cite{thurston-stretch} (see $\S 2$ for
   definition and discussion). We often denote $\g(t)$ simply by $X_t$.    
   The following theorem summarizes the main results of this paper which
   are analogues of some of the results of Rafi in
   \cite{rshteich,rcombteich,rteichhyper} about the behavior of \Teich
   geodesics. Also some similar results about the behavior of
   Weil-Petersson geodesics in \cite{bmm1,bmm2,wpbehavior, asympdiv, wpbtneck}. 
    
   \begin{theorem}\label{thm:main} 
     
     Let $\g \colon [a,b] \to \T(\s)$ be a geodesic in the Thurston metric
     with maximally stretched lamination $\lambda_\g$. There are positive
     constants $\rho$ and $K$ such that, for any non-annular subsurface $R
     \subseteq \s$ that intersects $\lambda_\g$, there is a possibly empty
     interval of times $J_R=[c,d] \subset [a,b]$, with the following
     properties: 
     \begin{enumerate} 
       \item For any $t \in J_R$, we have $\ell_t(\bd{R}) \le \rho$. 
       \item If $s,t \in I$ are on the same side of $J_R$ ($s,t<c$ or
         $s,t>d$), then \[ d_R(X_r, X_t) \le K.\] 
     If $J_R$ is empty, this holds for every $s, t \in I$.  
       \item The shadow of $\g$ to $R$, $\Upsilon_R \circ \g \big|_{J_R} :J_R \to \calC(R)$, is a
         reparametrized quasi-geodesic in $\calC(R)$. 
     \end{enumerate}
   \end{theorem}
  
   The key assumption in the above theorem is that $R$ intersects 
   $\lambda_\g$. This restricts the freedom of changing the structure of 
   subsurface $R$ arbitrarily. In particular, if $\lambda_\g$ is a filling lamination
   (meaning it intersects every essential simple closes curve on $S$), 
   the shadow of a Thurston geodesic to every surface is a reparametrized quasi-geodesic.
   At the end of the paper, we will show that the assumption that $\lambda_\g$ is
   filling can be replaces a assumption that $\lambda_\g$ is nearly filling   
   (see  \defref{Nearly-filling}) which should be interpreted to mean, from the point of 
   view of the hyperbolic metric $X$, $\lambda_\g$ looks like a filling lamination. 
   
     \begin{theorem}\label{Thm:filling}

    Let $\g \colon [a,b] \to \T(\s)$ be a Thurston geodesic. If $\lambda_\g$ is
    nearly filling on $X_a$, then for any subsurface $R \subseteq \s$,
    the set $\Upsilon_R(\g([a,b]))$ is a reparametrized quasi-geodesic in
    $\calC(R)$.
    
  \end{theorem}

  The main technical tool of this paper is to further develop the notion of
  a horizontal curve along a Thurston geodesic. This notion was first
  introduced in \cite{LRT2}, in analogy with the existing picture for the
  behavior of a curve along a \Teich geodesic, as developed in
  \cite{mm1,rcombteich,rteichhyper}. In essence, a curve $\alpha$ is
  horizontal if it spends a significant portion of its length fellow
  traveling $\lambda_\g$. It was shown in \cite{LRT2} that horizontal
  curves stay horizontal (in fact get more horizontal) along a Thurston
  geodesic and the length of the portion of $\alpha$ that fellow travels
  $\lambda_\g$ grows fast. We generalize this to a curve inside a
  subsurface (\thmref{Horizontal}). We then establish a trichotomy
  (Proposition \ref{Prop:Trichotomy}) stating that, at every time $t$ and
  for every subsurface $R$ either $\partial R$ is horizontal in $X_t$,
  $\partial R$ is short at $X_t$ or $\partial R$ is vertical meaning that
  the projection of $X_t$ and $\lambda_\g$ to $\calC(R)$ are close. This
  implies that the only time the shadow of $\g$ to $R$ can change is when
  $\partial R$ is short. 

  \subsection*{Notations}
  
  Here we list some notation and conventions that are used in this paper. 
  We will call a constant $C$ \textit{universal} if it depends only on the
  topological type of  a surface. Given two quantities $a$ and $b$,  we
  will write $$a\lmul b$$ if there is a universal constant $C$ such that
  $a\leq Cb$. 
 
  Similarly, we write $$a\ladd b$$ if  $a\leq b+C$ for a universal constant $C$.

  We write $a\emul b$ if $a\lmul b$ and $b\lmul a$. Likewise,  $a\eadd b$
  if $a\ladd b$ and $b\ladd a$.

  
  We define the logarithmic cut-off function $ \Log:\mathbb{R}^+\to [1,\infty)$ as 
   \begin{equation}\label{eqn:Log}
   \Log x:=\max\Big\{\ln x, 1\Big\},
  \end{equation} 
  where $\ln x$ is the  natural logarithm of $x$.

  For the convenient of the reader, we also list the named constants that are used throughout the paper. 
  \begin{itemize}
    \item $\ep_B$ = Bers constant
    \item $\ep_w$ = weakly-horizontal threshold constant
    \item $\ep_h$ = horizontal threshold constant
    \item $\rho$ = boundary of $R$ threshold constant
    \item $s_0$   = Theorem \ref{Thm:Horizontal} 
    \item $n_0$   = horizontal intersection constant 
    \item $L_0$   = horizontal fellow traveling constant
    \item $w_\ep$ = .5 width of the standard collar about a curve of $\ep$
      length.
    \item $b_\ep$ = boundary length of the standard collar about a curve of
      $\ep$ length
    \item $L_\ep$ = Lemma \ref{Lem:Long}
    \item $A$ = Lemma \ref{Lem:ConstantA}
    \item $B$ = Lemma \ref{Lem:ConstantB}
    \item $C(n,L)$ = Lemma \ref{Lem:Corridor1}
    \item $M$ = Lemma \ref{Lem:ConstantM}
  \end{itemize}

   \subsection*{Acknowledgments} 

   The third named author was partially supported by NSERC Discovery grant
   RGPIN 05507. The last author was partially supported by NSF DMS-1651963.
  
\section{Background}

  \subsection{Metric spaces and coarse maps} 
  \label{subsec:cc}

  Let $C>0$, and let $\calX$ and $\calY$ be two metric spaces.  A multivalued map
  $\phi:\calX\to \calY $ is called a $C$--\textit{coarse map} if for every
  $x \in \calX$,  
  \[\diam_\calY\big(\phi(x)\big) \le C.\] 
  Moreover, for $K\geq 1$, we will say it is  $K$--\textit{Lipschitz} if in addition, for all $x, x' \in \calX$, 
  \[\diam_\calY\big( \phi(x)\cup\phi(x') \big) \le Kd_\calX(x,x').\] 
  Now, let $\calY$ be a subspace of $\calX$ equipped with the
  induced metric. We say a $K$--coarse map $\Pi \colon \calX \to \calY$ is a $K$--\textit{retraction}
  if, for all $y \in \calY$, 
  \[\diam_{\calX}\big(\Pi(y) \cup \set{y}\big) \le K.\]
  
  
  By a \emph{path} in $\calX$ we will mean a coarse map $\phi \colon [a,b]
  \to \calX$. A path $\phi \colon [a,b]  \to \calX$ is a $K$--quasi-geodesic
  if for all $a\le s \le t \le b$, \[ \frac{1}{K}(t-s) -K \le \diam_X \big(
  \phi(s)\cup \phi(t) \big) \le K(t-s) + K.\] We will say $\phi$ is a
  \emph{reparametrized} $K$--quasi-geodesic if there is an increasing
  function $h \colon [0,n] \to [a,b]$ such that $\phi \circ h$ is a
  $K$--quasi-geodesic. 
  When $\calX$ is $\delta$--hyperbolic, then a path $\phi$ is a reparametrized
  $K$--quasi-geodesic for some $K$ if an only if there 
  exists $K'$ such that for all $r,s,t \in [a,b]$, where $r \le s \le t$,
  we have
  \[ 
  \diam_{\calX}(\phi(r) \cup \phi(s)) + \diam_{\calX}(\phi(s) \cup \phi(t))
  \le K' \diam_{\calX}(\phi(r) \cup \phi(t)) + K'. 
  \] 




 

  The following theorem is a standard fact in the coarse geometry; see
  e.g.\ the proof of \cite[Theorem 5.7]{LRT2}. 

  \begin{theorem} \label{Thm:Reparam}

    Let $K\geq 1$ and suppose that $\phi$ is a path in $\calX$ and that there
    exists a (coarse) $K$--Lipschitz retraction map $\Pi \from \calX \to
    \text{Im}(\phi)$. Then $\phi$ is a reparametrized $K'$--quasi-geodesic
    for a $K'\geq 1$ that depends only on $K$. 

  \end{theorem}
 
  In this section we will discuss arcs, curves and subsurfaces on a
  compact, connected, oriented surface \s of genus $g$ with possibly $b$
  punctures or boundary components. The {\em complexity} of \s is defined
  by $\xi(\s):=3g-3+b$. 
    
  A \emph{curve} on \s is a closed connected $1$--manifold on \s defined up
  to free homotopy. We always assume that a curve is essential i.e.\
  non-trivial and non-peripheral.  We say two curves are disjoint if the
  curves have disjoint representatives. A \emph{multicurve} is a union of
  pairwise disjoint curves. An \emph{arc} on \s is a properly embedded
  simple $1$-manifold with boundary where the endpoints lie on $\bd \s$
  defined up to homotopy relative to $\bd \s$. We always assume that an arc
  is not boundary parallel. A \emph{subsurface} $R\subseteq \s$ is a
  compact, connected $2$--dimensional submanifold of $\s$ with boundary
  such that each boundary component of $R$ is either an essential curve or
  a boundary component of \s defined up to homotopy relative to $\partial
  \s$. In particular, the map $\pi_1(R) \to \pi_1(S)$ induced by inclusion
  is injective. 

  Two curves or arcs $\alpha,\beta$ \emph{intersect} if the curves are not
  disjoint. We denote by $\I(\alpha,\beta)$ the minimal number of
  intersections between representatives of $\alpha$ and $\beta$. A curve
  $\alpha$ \emph{essentially intersects} a subsurface $R$ if $\alpha$ does not have a
  representative that is disjoint from $R$; equivalently, $\alpha$ is
  either a curve in $R$ or it intersects a component of $\bd R$. Finally,
  two subsurfaces intersect if they do not have disjoint representatives. 

  \subsection{Pants decompositions and markings} 

  A \emph{pants decomposition} on \s is a multicurve whose complementary
  regions are all 3-holed spheres (or equivalently a maximal set of
  pairwise disjoint curves on $S$); the number of curves in a pants
  decomposition is equal to $\xi(\s)$. A \emph{marking} $\mu$ on \s is a
  pants decomposition $P = \{\alpha_i\}$ together with a set $\overline{P}
  = \{ \balpha_i \}$ of curves such that, for each $i$, $\balpha_i$
  intersects $\alpha_i$ minimally, and $\balpha_i$ does not intersect
  $\alpha_j$ for all $j \ne i$; $\balpha_i$ is called the \emph{transverse
  curve} to $\alpha$. We will usually think of a marking $\mu$ as the set
  of the curves in $P \cup \overline{P}$. 
  
  \subsection{Geodesic laminations}

  For the rest of this paper, we will assume that \s is a surface with
  $\xi(\s) \ge 1$, and by a hyperbolic metric on \s we will mean a complete
  finite-area hyperbolic metric on the interior of \s. 
  
  Now fix a hyperbolic metric on \s. A \emph{geodesic lamination} $\lambda$
  on \s is a closed subset of $S$ that is a disjoint union of simple
  complete geodesics.  These geodesics are called the \emph{leaves} of
  $\lambda$. The space of geodesic laminations, equipped with Hausdorff
  topology, is compact. Moreover, any two different hyperbolic metrics on
  \s determine homeomorphic spaces of geodesic laminations, so the space of
  geodesic laminations $\Lam(\s)$ depends only on the topology of \s.  (see
  \cite{Penner:ttr}, \cite{thurston-stretch} \cite{Bon:lam} for details). 
  
  
  \subsection{Curve graphs and subsurface projection}

  Let $R\subseteq \s$ be a subsurface. The \emph{arc and curve graph}
  $\calC(R)$ of $R$ is defined as follows: If $R$ is non-annular, then
  vertices of $\calC(R)$ are the arcs and curves on $R$ and two vertices
  are connected by an edge if they do not intersect. When $R$ is an
  annulus, the definition of $\calC(R)$ is a bit more involved. In this
  situation, the vertices of the complex are the homotopy classes of arcs
  that connect the boundaries of the natural compactification of the
  annular cover of $S$ corresponding to $R$, and edges correspond to arcs
  with disjoint interiors. 
  

  Assigning length one to each edge turns $\calC(R)$ into a metric graph
  which is Gromov hyperbolic \cite{mm1}. We represent this metric by
  $d_R(\param,\param)$.
    
  In the following we define the \emph{subsurface projection map}  and {\em
  subsurface projection coefficients} that play a crucial role in
  description of our results in the paper. 
  \begin{definition} \label{Def:pi}
  The subsurface projection map
  \[ \pi_R: \Lam (S)\to \calP(\calC(R)) \cup \set{\infty} \]
  is defined as follows. Fix a hyperbolic metric on \s so that $R$ is
  represented by a unique convex subsurface with geodesic boundary and any
  geodesic lamination $\lambda$ is uniquely represented by geodesics in the
  metric so that all intersections between these representatives and the
  subsurface $R$ are essential. Then
  \begin{itemize}
    \item If $\lambda$ does not intersect $R$, define
      $\pi_R(\lambda)=\emptyset$. 
    \item If $\lambda \cap R$ contains a curve or an arc, then define
      $\pi_R(\lambda)$ to be the set of all curves and arcs in $\lambda
      \cap R$ (as before up to homotopy). 
    \item If $\lambda \cap R$ has no compact segment, and there is an
      element $\alpha \in \calC(R)$ disjoint from $\lambda \cap R$, then
      define $\pi_R(\lambda)$ to be the set of all curves and arcs in $\calC(R)$ 
      disjoint from $\lambda$. 
    \item Finally, if $\lambda \cap R$ has no compact segment and every
      element of $\calC(R)$ intersects $\lambda$, then define $\pi_R(\lambda)
      = \infty$.
  \end{itemize}
  \end{definition}
      
  It is easy to see that $\diam_R(\pi_R(\lambda)) \le 2$ when
  $\pi_R(\lambda) \notin \set{\emptyset,\infty}$, so $\pi_R$ is coarsely
  well-defined. Note that the above definition is independent of the choice
  of the hyperbolic metric. Note that a curve or lamination $\lambda$
  intersects a subsurface $R$ essentially if $\pi_R(\lambda) \neq
  \emptyset$.
  
  The notion of subsurface projection generalizes to markings on a surface
  as follows. When $R$ is a non-annular subsurface, $\pi_R(\mu)$ is the
  projection of any curve in the base of $\mu$ that intersects $R$. When
  $R$ is an annulus with core curve $\gamma$ that in not contained in the
  base of $\mu$, $\pi_R(\mu)$ is the projection of any curve in the base
  that intersects $\gamma$. But, when $\gamma$ is contained in the base of
  the marking, we define  $\pi_R(\mu)$ to be the projection of the
  transversal curve $\bar \gamma\in \mu$ to $\gamma$.  
 
 
  Let $p,q$ be any pair of geodesic laminations or markings that intersect
  $R$ essentially.  If $\pi_R(p),\pi_R(q)\subset \calC(R)$, we define the
  subsurface coefficient by:
  \begin{equation}\label{eq:dR}
    d_R (p,q) :=  \min \Big\{ d_R(\alpha,\alpha') :\alpha\in
    \pi_R(p)\;\text{and}\; \alpha'\in\pi_R(q) \Big\}.
  \end{equation}
  If either $p$ or $q$ are mapped to $\infty$ by $\pi_R$, we define $$
  d_R(p,q)=+\infty. $$ The subsurface projection coefficients satisfy the
  triangle inequality with an additive error, that is for any three
  geodesic laminations or markings $p,q,r$ that intersect $R$ we have
  \begin{equation}\label{eq:triineq}
    d_R(p,q)\ladd d_R(p,r) + d_R(r,q)
  \end{equation}
  where the additive error can be taken to be $2$. 
 
  We also define the \emph{intersection number} of any two laminations or
  curves $\lambda$ and $\lambda'$ that intersect $R$ to be 
  \begin{equation}\label{eq:iR}
  \I_R(\lambda,\lambda') = 
   \min \Big\{ \I_R(\alpha, \alpha') : 
      \alpha \in \pi_R(\lambda),  \ \alpha' \in \pi_R(\lambda') \Big\}
   \end{equation}
  in the case when both $\pi_R(\lambda)$ and $\pi_R(\lambda')$ are
  contained in $ \calC(R)$, and $+\infty$ otherwise. Using a surgery
  argument \cite[\S 2]{mm1} we may see that 
  \begin{align}\label{Eqn:di}
    d_R(\alpha,\beta) \lmul \Log \I_R(\alpha,\beta).
  \end{align}
  where $\Log$ is the cut-off logarithmic function defined in (\ref{eqn:Log}).

  \subsection{The standard collar}
   
  Let  $X$ be a hyperbolic metric on $S$ and  $\alpha$ a curve on $X$. Then
  $\alpha$ has a unique geodesic representative and we will denote its
  length by $\ell_X(\alpha)$.
   
  For a  multicurve $\sigma$ let
  \begin{equation}\label{eq:len multicurve} 
  \ell_X(\sigma)=\max_{\alpha\in\sigma}\ell_X(\alpha).
  \end{equation}
  A multicurve $\sigma$ is $\ep$--\textit{short} on $X$ if $\ell_X(\sigma)
  \le \ep$. A subsurface $R$ in $X$ is $\ep$--\textit{thick} if $\ell_X(\bd
  R) \le \ep$ and every curve in $R$ has length at least $\ep$. 
  

  \begin{lemma}[Collar Lemma{\cite[\S 4.1]{buser}}]

    For any hyperbolic metric $X$ on $S$, if $\alpha$ is a  geodesic curve
    with $\ell_X(\alpha) = \ep$, then the regular neighborhood
    $U_X(\gamma)\subseteq X$ of $\alpha$ with width $2w_\ep$ where
     \begin{equation}\label{width collar}
      w_\ep = \sinh^{-1} \left(
    \frac{1}{\sinh(\ep/2)} \right),
    \end{equation}
     is an embedded annulus, and each
    boundary of $U_X(\alpha)$ has length 
    \begin{equation}\label{eq:len bd collar}
    b_\ep = \ep \coth (\ep/2). 
    \end{equation}
     If $\beta$ is a geodesic curve disjoint from $\alpha$, then
    $U_X(\beta)$ and $U_X(\alpha)$ are disjoint.
  \end{lemma}

   \begin{definition}
     We will call $U_X(\gamma)$ the \emph{standard collar} about $\gamma$
     on $X$. A boundary curve of a regular neighborhood of a closed
     geodesic is called a \emph{hypercycle}.
   \end{definition}


  \begin{remark} 
    It is easy to see that when $\ep<1$, $w_\ep$ is $\log\frac{1}{\ep}$ up
    to an additive error of at most $\log{3}$. Also, $b_\ep$ is an
    increasing function with $\lim_{\ep\to 0}b_\ep=2.$ 
  \end{remark}

  \subsection{Short marking} \label{Sec:marking}
  Recall that the \emph{Bers constant} is a constant $\ep_B$ that depends
  only on the complexity of $S$ \cite[\S 5]{buser}, such that any
  hyperbolic metric $X$ on \s admits a pants decomposition $\calP$ where
  $\ell_X(\partial \calP) \le \ep_B$. By increasing $\ep_B$ if necessary, we
  may assume that $\ep_B \ge 1$ has the property that whenever
  $\ell_X(\alpha) \ge 1$, then the shortest transverse curve to $\alpha$ is
  $\ep_B$--short.
 
  A pants decomposition $\calP$ with $\ep_B$-short boundary will be called
  a \textit{short pants decomposition}. A \textit{short marking} $\mu_X$ on
  $X$ is a marking consisting of a short pants decomposition as the base
  and that the transversal curve to each base curve is chosen to be the
  shortest possible on $X$. 
 
  For a subsurface $R\subseteq S$ we write $d_R(X,\param)$ for
  $d_R(\mu_X,\param)$. 

  \subsection{Twisting and length}
  
  \label{subsec:twisting}

  In the case when $R$ is an annulus with core curve $\gamma$, we use the
  notation $d_\gamma(X,\param)$ instead of $d_R(\mu_X,\param)$.  For a
  lamination $\lambda$, we will call $d_\gamma(X,\lambda)$ the
  \emph{relative twisting} of $\lambda$ about $\gamma$ on $X$.
  
  The following lemma provides us with some useful estimates for the length
  of an arc inside the collar neighborhood of a curve $\gamma$ and the
  twisting of the arc about $\gamma$ on $X$.
  
  \begin{lemma} \label{Lem:ConstantA}
    
    There is a universal constant $A>0$ such that the following statements
    hold. Let $X$ be a hyperbolic metric on \s admitting a curve $\gamma$
    of $\ell_X(\gamma) =1$. Then for any simple geodesic $\alpha$
    intersecting $\gamma$, the length $L(\alpha)$ of the longest arc in
    $\alpha \cap U_X(\gamma)$ has $$d_\gamma(X,\alpha) -A \le L(\alpha) \le
    d_\gamma(X,\alpha) + A.$$
  
    \end{lemma}
   
  \begin{proof}

    It follows from \cite[Lemma 3.1]{LRT2} that
    \[\Big|L(\alpha)-\I(\tau_\gamma,\alpha_\gamma)\ell_X(\gamma)\Big|\ladd
    w_1\] where $\tau_\gamma\subset U_X(\gamma)$ is a geodesic arc
    crossing $U_X(\gamma)$ and perpendicular to $\gamma$, $\alpha_\gamma$
    is a component  of $\alpha\cap U_X(\gamma)$ and $w_1$ is half of the
    width of the standard collar of $\gamma$. Since
    $\I(\alpha_\gamma,\tau_\gamma)\eadd d_\gamma(\alpha,X)$,
    $\ell_X(\gamma)=1$ and $w_1$ is a fixed number, the lemma follows.
    \qedhere 
  
  \end{proof}

  \subsection{Thurston metric} 

  \label{Sec:ThurstonMetric} 
     
  In this subsection we recall the definition of a metric on \Teich space
  introduced by Thurston \cite{thurston-stretch}.

  Recall that the \Teich space of \s, denoted by $\T(\s)$, is the space of
  hyperbolic metrics on \s modulo the action of homeomorphisms of \s that
  are isotopic to the identity. The space $\T(\s)$ is homeomorphic to an
  open ball of dimension $2 \xi(\s)$ and has been the subject of study in
  particular in conjunction with the mapping class groups of surfaces.

  Given $X,Y \in \T(S)$, let $L(X,Y)=\inf_f L_f$, where $f \from X \to Y$
  is a Lipschitz homeomorphism from $X$ to $Y$ isotopic to the identity on
  \s, and $L_f$ is the Lipschitz constant of $f$. A map $f \from X \to Y$
  is called {\em optimal } if $L_f = L(X,Y)$. Thusrston \cite[Theorem
  8.5]{thurston-stretch} showed that such a map exists and defined
  \emph{Thurston metric} on $\T(\s)$ by
  \begin{equation}\label{Th}
    \dth(X,Y) := \log L(X,Y).
  \end{equation}

  Thurston also introduced a geodesic lamination $\lambda(X,Y)$ called
  {\em maximal stretch lamination}  where for any segment
  $\sigma\subseteq \lambda(X,Y)$ with finite length we have that
  $\len(f(\sigma))=L\cdot\len(\sigma)$. In the following theorem we
  collect some of the important properties of Thurston metric.

  \begin{theorem}[Thurston metric]\label{Thm:Thurston}
    \cite[Theorem 8.2, Theorem 8.5]{thurston-stretch}
    
    The Thurston metric $d(\param,\param)$ is an asymmetric complete
    geodesic metric on $\T(\s)$. Moreover, for all $X, Y \in \T(\s)$,  

    \begin{enumerate}
      \item $L(X,Y) =  \sup_{\alpha\in \calC(S)}
        \frac{\ell_Y(\alpha)}{\ell_X(\alpha)}$. \label{th:sup}
      \item There exists an optimal map from $X$ to $Y$ and every optimal
        map preserves the leaves of $\lambda(X,Y)$, stretching the length
        of each leaf by a factor of $e^{\dth(X,Y)}$. \label{th:exp}
      \item For any Thurston geodesic $\g \colon I \to \T(S)$, $I \subseteq
        \RR$, there is a geodesic lamination $\lambda_\g$, maximal with
        respect to inclusion, such that for all $t > s$, $\lambda_\g
        \subseteq \lambda(\g(s),\g(t))$. Moreover, if $I=[a,b]$ is a finite
        interval, then $\lambda_\g = \lambda(\g(a),\g(b))$. 
    \end{enumerate}
  \end{theorem}

  \subsection{Shadow map}

  Let $R \subseteq \s$ be a subsurface and let $\Upsilon_R \colon \T(\s)
  \to \calC(R)$ be the coarse map defined by $\Upsilon_R(X) = \pi_R(\mu_X)$,
  where $\mu_X$ is a short marking on $X$. The final result
  of this section is that $\Upsilon_R$ is a Lipschitz map.

  \begin{lemma} \label{Lem:Lip} 
    
    For any subsurface $R \subseteq \s$, the coarse map $\Upsilon_R$ is
    $K$--Lipschitz for some $K\geq 1$ that depends only on \s.   

  \end{lemma}
  
  \begin{proof} 
    
    Suppose $\dth(X,Y) = 1$ and $\alpha$ and $\beta$ are $\ep_B$--short
    curves on $X$ and $Y$ respectively both intersecting $R$. Since the
    length of a curve grows at most exponentially in the distance from $X$
    to $Y$, we have $\ell_Y(\alpha) \le e \ep_B$. Therefore,
    $\I_R(\alpha,\beta) = O(1)$, which implies $d_R(X,Y) = O(1)$. \qedhere 
  
  \end{proof}
 
\section{Tools from hyperbolic geometry}  

  In this section we gather some useful properties of the hyperbolic plane
  and hyperbolic surfaces. Some of the results are elementary and are known
  in spirit, but to our knowledge the statements do not directly follow
  from what is known in literature.  
  
  \begin{lemma} \label{Lem:closer} 
    For any $0 < \ep_1 < \ep_2$, set $L = \log 2
    \left(\frac{\sinh{\ep_2}}{\sinh{\ep_1}} \right).$ Let $\lambda$ and
    $\alpha$ be two geodesics in $\HH$ and let $p_0, p_1, p_2$ be 3 points
    on $\lambda$ appearing in that order. Assume that,
    \[
    d_\HH(p_0, p_1) \geq L, \qquad 
    d_\HH(p_1, p_2) \geq L, \qquad 
    d_\HH(p_0, \alpha) \leq \ep_2, \qquad 
    d_\HH(p_2, \alpha) \leq \ep_2. \qquad 
    \]
    Then 
    \[
    d_\HH(p_1, \alpha) \leq \ep_1. 
    \]
  \end{lemma}  
  
  \begin{proof}
    Let $a$ and $b$ be points in $\HH$ on the same side of $\lambda$ and
    such that  segments $[a,p_0]$ and $[b,p_2]$ are both perpendicular to
    $[a,b]$ and $d_\HH(p_0,a)=d_\HH(p_2,b)=\epsilon_2$. Let $c\in\lambda$
    be the closest point to $[a,b]$, clearly it is also the midpoint of
    $[p_0,p_2]$. We have $d_\HH(p_0,c)= \frac 1 2 d_\HH(p_0,p_2)\geq L$ and
    hence $$\sinh d_\HH(c,[a,b])=\frac{\sinh d_\HH(a,p_0)}{\cosh
    d_\HH(p_0,c)}< \frac{\sinh \ep_2}{\frac 1 2 e^L}=\sinh \ep_1.$$ This
    implies that some point $p'\in [p_0,c]$ is distance $\ep_1$ from
    $[a,b]$. We have $$\frac {\sinh \ep_2}{\sinh \ep_1}=\frac{\cosh
    d_\HH(p_0,c)}{\cosh d_\HH(p',c)}\geq \frac{1}{2}
    e^{d_\HH(p_0,c)-d_\HH(p',c)}$$ and so $$d_\HH(p',c)\geq
    d_\HH(p_0,c)-L.$$ Similarly, let $p''\in[c,p_2]$  which is distance
    $\ep_1$ from $[a,b]$. By the same computation $$d_\HH(p'',c)\geq
    d_\HH(p_2,c)-L.$$ This means in particular that $d_\HH(p_1,[a,b])\leq
    \epsilon_1$. 

    Let $[a',b']$ be the image of $[a,b]$ under the reflection in
    $\lambda$. Recall that $d_\HH(p_0,\alpha)\leq \ep_2$ and
    $d_\HH(p_2,\alpha)\leq \ep_2$. Hence $\alpha$ crosses the polygon
    $abp_2b'a'p_0$ at $[p_0,a]$ or $[p_0,a']$ and $[p_2,b]$ or $[p_2,b']$.
    We conclude that  $$d_\HH(p_1,\alpha)\leq
    d_\HH(p_1,[a,b])=d_\HH(p_1,[a',b'])\leq \epsilon_1. \qedhere$$
  \end{proof}
  
  \begin{lemma}\label{Lem:Intersect}
  
    Let $\alpha$ and $\lambda$ be disjoint geodesics in $\HH$. Let $a,a'\in
    \alpha$ be the endpoints  of the orthogonal projection of $\lambda$ to
    $\alpha$. Then for any hyperbolic isometry $\phi$ with axis $\alpha$
    and translation length $\ell < d_\HH(a,a')$, the geodesics
    $\phi(\lambda)$ and $\lambda$ intersect.   \end{lemma}
  
  \begin{proof} 

   Let $\phi$ be a hyperbolic isometry with axis $\alpha$ and translation
    length $\ell<d_\HH(a,a')$. Since $\lambda$ is disjoint from $\alpha$,
    the endpoints of $\lambda$ and $\phi(\lambda)$ are all on the same side
    of $\alpha$. The orthogonal projection of $\phi(\lambda)$ to $\alpha$
    has endpoints $\phi(a)$ and $\phi(a')$. The condition on $\ell$ implies
    that exactly one of these points is between $a$ and $a'$. Hence the
    endpoints of $\lambda$  are separated by one endpoint of
    $\phi(\lambda)$, so $\lambda$ and $\phi(\lambda)$ intersect. \qedhere
   
  \end{proof}
  
  \begin{lemma} \label{Lem:ConstantM} 

     There exists a universal constant $M>0$ such that the following holds.
     Let $X$ be a hyperbolic metric on $S$. Let $\gamma$ be a curve with
     $\ell_X(\gamma) = \ep\leq \ep_B$, and $\tau$ be a shortest curve
     intersecting $\gamma$. Then 
     \[ 2w_\ep\le \ell_X(\tau) \le 4w_\ep+2M. \] 

   \end{lemma}

   \begin{proof}
     
     Let $\calP$ be a short pants decomposition on $X$. If
     $\ell_X(\gamma)=\epsilon\leq \ep_B$, then either $\gamma$ intersects a
     curve in $\calP$ or $\gamma\in\calP$. In the former case, we have
     immediately that $\ell_X(\tau)\leq \ep_B$. In the latter case,
     $\gamma$ is contained in a subsurface $R\subseteq S$ which is either a
     one-holed torus or a four-holed sphere with boundary curves in
     $\calP\ssm \gamma$. 
     
     The diameter of the complement of half collar neighborhoods of the
     boundary curves of $R$ is uniformly bounded above by some $M>0$
     independent of $X$.  Thus
     there is a curve $\tau$ that intersects  $\gamma$ once or twice whose
     length is at least $2w_\ep$ and at most $4w_\ep+2M$. \qedhere

   \end{proof}

   \begin{lemma} \label{Lem:Long}

     For any $\ep\in(0, \ep_B)$, there exists $L_\ep>0$ such that the
     following statement holds. Let $X$ be a hyperbolic metric on \s and
     let $R$ be an $\ep$--thick subsurface. Let $\alpha$ be a simple closed
     geodesic contained in $R$. Then any segment of $\alpha$ of length at
     least $L_\ep$ must intersect some curve in a short marking $\mu_R$ on
     $R$.

  \end{lemma}
  
  \begin{proof}
  
    Let $Y \subseteq R$ be a pair of pants obtained by cutting $R$ along
    the pants curves in $\mu_R$. For each $\gamma \subset \partial Y$, let
    $U_Y(\gamma) = U_X(\gamma) \cap Y$, where $U_X(\gamma)$ is the standard
    collar neighborhood of $\gamma$. Each boundary component of $Y$ is
    either a pants curve in $\mu_R$ or a boundary curve of $R$. Since $R$
    is $\ep$--thick, the collar width of every pants curve in $\mu_R$ is at
    most $w_\ep$. For each $\gamma \subset \partial Y$ that is not in
    $\partial R$, let $\tau_\gamma$ be a component of $\bgamma \cap
    U_Y(\gamma)$, where $\bgamma$ is the dual curve to $\gamma$ in $\mu_R$.
    Now consider the truncated pants \[ \overline{Y} = Y \setminus
    \bigcup_{\gamma \subset \partial Y} U_Y(\gamma). \]

    Let $\alpha$ be a simple geodesic curve in $R$ and let $\omega$ be a
    segment of $\alpha$ contained in $Y$. Note that the segment $\omega$
    can enter each standard collar neighborhood at most once. Thus,
    $\omega$ can be divided into 3 pieces: 
    \[
      \omega = \omega_1 \cup \bomega \cup \omega_2,
    \]
    where $\bomega = \omega \cap \overline{Y}$, and for each $i=1,2$, there
    is a boundary curve $\gamma_i$  of $Y$ such that $\omega_i = \omega
    \cap U_Y(\gamma_i)$. Since $\alpha \subseteq R$, $\omega$ cannot enter
    the standard collar of any curve in $\partial R$, so each $\gamma_i$ is
    a pants curve of $\mu_R$. By \cite[Lemma 3.1]{LRT2}, we have
    \begin{align} \label{Eqn:Arc}
      \ell_X(\bomega) = O(1) 
    \quad
    \text{and} 
    \quad
      \Big| 
      \ell_X(\omega_i) - \I(\tau_{\bar\gamma_i},\omega)\ell_X(\gamma_i)
      \Big| 
      \ladd
      w_\ep. 
    \end{align}
    Now, if $\omega$ is long enough and it does not intersect a pants curve
    of $\mu_R$, then $\omega$ is contained in a pair of pants $Y$. By
    \eqnref{Arc} then $\I(\omega,\tau_{\bar\gamma_i})\geq 1$ for some
    $i=1,2$ and $\omega$ must intersect $\bar\gamma_i$ which is a curve of
    $\mu_R$. \qedhere 

  \end{proof}
  
  \begin{lemma}\label{Lem:len comp in collar}

    Let $X$ be a hyperbolic metric on $S$. Given a constant $b$ and let $U$
    be an annulus bounded by hypercycles with boundary lengths at most $b$.
    Suppose $\tau\subset U$ and $\eta\subset U$ are two arcs connecting the
    boundaries of $U$, where $\tau$ is a geodesic arc. Then
    \[
    \ell_X(\tau) \leq \ell_X(\eta)+(\I(\tau,\eta)+1) b.
    \]
    If, in addition, $\tau$ and $\eta$ intersect at least twice and the
    distance between any two consecutive intersection points  along $\tau$
    is at least  $D>3b$,  then \[\ell_X(\tau)\leq
    \frac{\ell_X(\eta)}{(1-\frac{3b}{D})}.\]
  \end{lemma}
  
  \begin{proof}
    Lifting to the universal cover, we can construct a path with the same
    endpoints  as $\tau$ which is a concatenation of a segment of a
    boundary component of $U$ of length at most $b$, followed by  $\eta$
    and a segment of the other boundary component of $U$ of length at most
    $\I(\eta, \tau)b$. Hence
    \begin{equation}\label{bb}
      \ell_X(\tau) \leq \ell_X(\eta)+(\I(\tau,\eta)+1) b
    \end{equation}
    which is  the first statement of the lemma. 
    
    Suppose  now that the distance between any two consecutive intersection
    points of $\tau$ and $\eta$ along $\tau$ is at least $D$.  We have
    $\ell_X(\tau)\geq D$ and $\I(\tau,\eta)-1\leq \frac{\ell_X(\tau)}{D}$.
    Incorporating this in \ref{bb} gives $$ \ell_X(\tau)\leq \ell_X(\eta)+
    \left( \frac{\ell_X(\tau)}{D}+2 \right) b \leq \ell_X(\eta)+\frac{3b} D
    \ell_X(\tau), $$ which implies the second statement of the lemma.
    \qedhere
  \end{proof}
  
  We proceed with the following lemma which is a modification of
  \cite[Lemma 3.2]{LRT2} about the length of a curve inside a  pair of
  pants.

  \begin{lemma} \label{Lem:Tooshort}
    For any $\ell>0$ there exists $D_0=D_0(\ell)$ such that the following
    holds. Let $X$ be a hyperbolic metric on $S$ and $\calP$ a pair of
    pants in $X$ with boundary lengths at most $\ell$. Let $\gamma$ be a
    simple closed curve (not necessarily a geodesic) that intersects
    $\partial \calP$ and let $\omega$ be a simple geodesic arc contained in
    $\calP$ with endpoints on $\gamma$ and has at least 2 intersections
    with $\gamma$ in the interior.  Write $\omega$ as a concatenation of
    $\omega_1, \omega_2$ and $\omega_3$ where $\omega_i$'s are segments of
    $\omega$ with endpoints on $\gamma$. Suppose that the distance along
    $\omega$ between any two consecutive intersection points with $\gamma$
    is at least $D_0$. Then for at least one $i\in\{1,2,3\}$ we have
    $$\ell_X(\omega_i)\leq 6\, \ell_X(\gamma).$$
  \end{lemma}
  
  \begin{proof}

    The proof is similar to that of \cite[Lemma 3.2]{LRT2}, but requires
    more care when it comes to additive and multiplicative errors. 
    
    Let $\alpha_i$, $i=1,2,3$, be the boundary components of $\calP$ and
    $U_i\subseteq \calP$ the standard collar of $\alpha_i$ in $\calP$. By
    (\ref{eq:len bd collar}), the length of the non-geodesic boundary of
    $U_i$ is bounded above by  $b=\ell \coth(\ell/2)$. Let $M>0$ be the
    length of the longest simple geodesic arc in $\calP$ that is disjoint
    from $U_i$'s. Note that $M$ depends only on $\ell$. 
   
    Let  $D_0\geq \max\{3M,6b\}$. We claim the following holds. 
    \begin{claim} There are $i,j\in \{1,2,3\}$ such that the following holds. 
    There is an annulus $U\subset U_j$ bounded by hypercycles, a subarc
      $\tau$ of $\omega_i$ contained in $U$ with $$\ell_X(\tau)\geq
      \frac{1} {3}\ell_X(\omega_i),$$ a subarc $\eta\subset U$ of $\gamma$
      that connects the boundary components of $U$, so that one of the
      endpoints of $\tau$ is on $\eta$.
    \end{claim} 
    \begin{proof}
      It follows from \cite[Lemma 3.1]{LRT2} that $\omega_2$ consists of
      (possibly empty) an arc  in some $U_j$ followed by an arc in
      $\calP\setminus\cup_i U_i$ followed by an arc  in some $U_k$. Hence
      by the assumption on $D_0$, either the first or the last arc has
      length at least $\frac{1}{3}\ell_X(\omega_2)$; we call this arc
      $\zeta$. Let $j$ be such that $\zeta\subseteq U_j$. The arc $\zeta$
      shares an endpoint $p$ with some $\omega_k\subseteq U_j$ for $k=1$ or
      $k=3$.  Now we have a division of $U_j$ into two annuli with disjoint
      interiors, one containing $\zeta$ and the other containing
      $\omega_k$. Since $\gamma$ passes through $p$, it crosses at least
      one of the annuli. We denote the annulus by $U$ and the subarc
      ($\omega_k$ or $\zeta$) of $\omega$ that is contained in $U$ by
      $\tau$. Clearly the length of $\tau$  is at least a third of the
      length of the arc $\omega_i$ that contains it and the claim follows.
     \end{proof}
     
     We now finish the proof of the lemma. Let $\tau$, $\eta$, $\omega_i$ and
     $U$  be as in the claim. The length of either of the boundary
     components of $U$ is at most $b$. If $\tau$ and $\eta$ intersect at
     least twice, we apply Lemma \ref{Lem:len comp in collar} and $D_0\geq
     6b$ to get
     \[
       \ell_X(\omega_i)\leq 3\ell_X(\tau)\leq
       \frac{3\ell_X(\eta)}{(1-\frac{3b}{D_0})}\leq
       \frac{3\ell_X(\gamma)}{(1-\frac{3b}{D_0})}\leq 6\,\ell_X(\gamma).\] 
     If $\tau$ and $\eta$ intersect once, then $\ell_X(\tau)\leq
     \ell_X(\eta)+b$. Further, since each $\omega_i$ is at least $D_0$
     long,
    $\ell_X(\tau)\geq \frac 1 3 D_0\geq 2b$ and so $\ell_X(\tau)\leq
    2\ell_X(\tau)$ which implies
    \[
    \ell_X(\omega_i)\leq 6\, \ell_X(\gamma). \qedhere
    \]
  \end{proof}
   
  We now introduce another useful tool, which is the notion of a \emph{thin
  quadrilateral} with long sides lying along the boundary of some proper
  subsurface $R\subset S$. In a lemma below, we will show the existence of
  such quadrilaterals when the boundary of $R$ is sufficiently long.

  \begin{definition}[Quadrilateral] 
    
    Let $R \subset \s$ be a proper subsurface. By a \emph{quadrilateral}
    $Q$ in $R$ we mean the image of a continuous map $f \colon [0,1] \times
    [0,1] \to R$, where $f$ is an embedding on $(0,1) \times (0,1)$, the
    edges $[0,1] \times \{0\}$ and $[0,1] \times \{1\}$ are mapped to $\bd
    R$, and all other points are mapped to the interior of $R$. We will
    call the images of $[0,1] \times \{1\}$ and $[0,1] \times \{0\}$ the
    {\em top} and {\em bottom} edges, respectively, and $\{0\} \times
    [0,1]$ and $\{1\} \times [0,1]$ the \emph{side edges} of $Q$. Also,
    note that top and bottom edges of the quadrilateral may not be
    distinct. By \emph{width} of $Q$ we will mean the maximal length over
    the top and bottom edges. We say $Q$ is \emph{$\delta$--thin} if the
    side lengths are at most $\delta$ and $Q$ is \emph{$L$--wide} if its
    width is at least $L$.

  \end{definition} 
  
  \begin{lemma}\label{Lem:Quad} 
    
    Given $\delta \in (0,1)$, let $L = \frac{2\pi|\chi(\s)|}{\delta}$ and
    $K = \frac{|\chi(\s)|}{\sinh^2(\delta/2)}$. Let $X$ be a hyperbolic
    metric on \s and $R\subset S$ a proper non-annular subsurface with
    $\ell_X(\bd R) > 2KL$. Then there is a non-empty collection of pairwise
    disjoint $\delta$--thin quadrilaterals in $R$ with total width at least
    $\ell_X(\bd R)/4$, and at least one among them is $L$--wide. Moreover,
    we can ensure this collection consists of at most $K$ quadrilaterals.

  \end{lemma} 
    
  \begin{proof} 
    
    Let $\tX \cong \HH$ be the universal cover of $X$. Let $\alpha$ be a
    curve in $\bd R$ with $\ell_X(\alpha) = \ell_X(\bd R)$, then let $U$ be
    the set of points $p \in \alpha$ such that for any lift $\tp$ of $p$ to
    $\tX$, $B_\delta(\tp)$, the ball of radius $\delta$ centered at $\tp$,
    meets exactly one component of the lifts of $\bd R$, i.e.\ the one that
    contains $\tp$. In this case, $B_\delta(p) \cap R$ is an embedded
    half-disk in $R$. By definition, for each arc $\omega$ in $\alpha \ssm
    U$, the $\delta$--collar neighborhood of $\omega$ is either a cylinder
    in $R$ and hence it is a $\delta$--thin quadrilateral in $R$ where the
    top and bottom edges are the same, or it contains an embedded
    $\delta$--thin quadrilateral in $R$. Moreover, the width of the
    quadrilateral containing $\omega$ is at least $\ell_X(\omega)$. Now
    consider the collection $\calQ$ of $\delta$--thin quadrilaterals
    containing an arc of $\alpha \ssm U$. 
      
    First note that the set $U$ has at most $K$ connected components. This
    follows from the fact that each connected component of $U$ has a
    $\delta$--collar neighborhood inside $R$ that contains at least one
    half-disk, and thus it contributes at least $2\pi \sinh^2(\delta/2)$
    (the area of half-disk) to the area of $R$. But the total area of $R$
    is at most $2\pi |\chi(S)|$ (area of the hyperbolic metric $X$), so $U$
    has at most $K$ components. This also implies that $\calQ$ consists of
    at most $K$ quadrilaterals. Moreover, using a similar area argument, we
    see that each connected component of $U$ has length at most $L$;
    because the $\delta$--collar neighborhood of $\omega\subset U$ in $R$
    has area at least $\ell_X(\omega) \delta$ which is bounded by
    $2\pi|\chi(S)|$.  Putting these observations and the assumption of the
    lemma about the length of $\alpha$ together  we get the lower bound
    \[
    \ell_X(\alpha) - KL >  \ell_X(\alpha)/2,
    \] 
    for the total length of components of $\alpha\ssm U$. Since $U$ has at
    most $K$ components, so does $\alpha\ssm U$. With each component of $U$
    having length at most $L$, at least one component of $\alpha \ssm U$
    has length at least $\ell_X(\alpha)/2K > L$. This implies that there is
    at least one $\delta$--thin and $L$--wide quadrilateral in $\calQ$.
    Finally, since the total length of all the arcs in $\alpha \ssm U$ is
    at least $\ell_X(\alpha)/2$ and each quadrilateral in $\calQ$ contains
    at most $2$ arcs in $\alpha \ssm U$, the total width of all the
    quadrilaterals is at least $\ell_X(\alpha)/4$. This finishes the proof
    of the lemma. \qedhere   
  \end{proof} 
    
  We now prove a proposition that allows us to estimate the relative
  twisting of a curve in a subsurface $R$ of $S$ with respect to a
  hyperbolic metric on $S$. 
  
  \begin{proposition} \label{Prop:RelTwist} 
    
    Let $X$ be a hyperbolic metric on \s. Let $R\subseteq S$ be a
    non-annular subsurface and $\gamma$ a curve intersecting $R$. Then for
    any component $\tau \subseteq \gamma \cap R$, we have
     \[ d_R(\gamma,X) \lmul \Log
    \frac{\ell_X(\tau)}{\ell_X(\bd R)},\] 
    where the error in the bound above depends only on \s. 
  
  \end{proposition}

  \begin{proof}
    
    Let $\beta$ be an $\ep_B$--short curve intersecting $R$, and let $\tau
    \subseteq \gamma \cap R$ and $\tau' \subseteq\beta \cap R$ be curves or
    arcs with endpoints on the boundary of $R$.

    For any $\delta \in (0,1)$, recall the constants $K$ and $L$ from Lemma
    \ref{Lem:Quad}. Now fix a small $\delta$ so that $8\delta^2 \le 2\pi
    |\chi(\s)|$ and $L -2\delta > \ep_B$. We consider the following two
    cases.
  
    First suppose that $\ell_X(\bd R) \le 2KL$. Recall that $w_{\ep_B}$
    denotes the width of the standard collar neighborhood about a curve of
    length $\ep_B$. Since $\beta$ is $\ep_B$--short, the distance between
    any two intersection points with $\beta$ along $\tau$ is at least
    $2w_{\ep_B}$, so we immediately have that
    \[ \I_R(\tau,\tau') \le \frac{\ell_X(\tau)}{2w_{\ep_B}} +1 \le
    \frac{KL}{w_{\ep_B}}
    \frac{\ell_X(\tau)}{\ell_X(\bd R) }+1. 
    \] 
    If the quantity $ \frac{KL}{w_{\ep_B}} \frac{\ell_X(\tau)}{\ell_X(\bd
    R) }$ is less than 1, then $\tau$ and $\tau'$ intersect at most once.
    If it is bigger than 1, we can write 
    \[\I_R(\tau,\tau') \leq 2\frac{KL}{w_{\ep_B}}
    \frac{\ell_X(\tau)}{\ell_X(\bd R) }\] The constant
    $2\frac{KL}{w_{\ep_B}}$ is independent of $X$ and the inequality
    (\ref{Eqn:di}) completes the proof in the case of $\ell_X(\bd R)\leq
    2KL$.
  
    Now assume that $\ell_X(\bd R) > 2KL$. Let $Q_1$,\ldots,$Q_m$, $m \le
    K$, be the collection of $\delta$--thin quadrilaterals in $R$ as
    guaranteed by Lemma~\ref{Lem:Quad}. Set $w_i$ to be the width of $Q_i$
    and assume that $w_1 > L$. Note that the side edges of all the $Q_i$'s
    are pairwise disjoint. Moreover, any geodesic that crosses $Q_i$ from
    side to side must be at least $w_i-2\delta$ long. Thus $\tau'$, which
    is an arc of an $\ep_B$--short curve, cannot cross $Q_1$ from side to
    side, so $d_R(\tau',\eta_1) = 1,$ where $\eta_1$ is a side edge of
    $Q_1$ missed by $\tau'$. Now let $n_i$ be the number of times that
    $\tau$ crosses $Q_i$ from side to side  for $i=1,\ldots,m$, and let
    $n=\min_{i=1,\ldots, m} n_i$. Then observe that 
    \begin{equation}\label{eq:dRlogn}
    d_R(\gamma,X) \lmul \Log n,
    \end{equation}
    holds. To see this, note that if $n=0$, then $d_R(\gamma,X)\leq 1$ and
    we already have the above inequality. So assume that $n\geq 1$  and let
    $\eta$ be the side edge that $\I_R(\tau,\eta) = n$, then by the
    triangle inequality \ref{eq:triineq} and \ref{Eqn:di} we have
    \begin{align*}
      d_R(\gamma,X) 
      & \ladd d_R(\tau,\eta) + d_R(\eta, \eta_1) +
      d_R(\eta_1,\tau') \\
      & \le d_R(\tau,\eta) + 2 \\ 
      & \lmul \Log \I_R(\tau,\eta) = \Log n, 
    \end{align*}
    again giving us the inequality (\ref{eq:dRlogn}). 
    
    We know that $\tau$ crosses each $Q_i$ from side to side at least $n$
    times, and picks up a length of at least $w_i-2\delta$ each time.
    Moreover, by Lemma \ref{Lem:Quad} we have that $\sum_{i=1}^m w_i >
    \ell_X(\bd R)/4$, so 
    \[ 
      n \le \frac{\ell_X(\tau)}{\ell_X(\bd R)/4 - 2m \delta}
      \le \frac{\ell_X(\tau)}{\ell_X(\bd R)/4 - 2K \delta}.
    \] 
    Also by our assumption, $8\delta^2 \le 2\pi|\chi(\s)|$ which implies
    that $2K\delta \le  \frac{1}{4} KL$, and hence by the lower bound for
    the length of $\bd R$ we have
    \[ \frac{\ell_X(\bd R)}{4} - 2K\delta \ge
    \frac{\ell_X(\bd R)}{8}. \] From the above two inequalities we deduce
    that $\Log n \lmul \Log \frac{\ell_X(\tau)}{\ell_X(\bd R)}$ which is
    the desired inequality. \qedhere
    
  \end{proof}

  \begin{lemma} \label{Lem:Collar2}

    Let $\lambda$ and $\lambda'$ be a pair of  disjoint non-asymptotic
    geodesics in $\HH$. Let $f \colon \HH\to \HH$ be a $K$--Lipschitz
    map, $K \ge 1$, such that $f(\lambda)$ and $f(\lambda')$ are geodesics,
    and $f$ stretches distances along $\lambda$ by $K$. Let $q \in
    f(\lambda)$ and $q' \in f(\lambda')$ be the endpoints of the common
    perpendicular between $f(\lambda)$ and $f(\lambda')$. Then for any
    points $x, y \in \lambda$, if $d_\HH(f(x),q) \ge d_\HH(f(y),q)$, then
    for any $x' \in \lambda'$, we have \[ \sinh d_\HH(q,q') \cosh
    \frac{K d_\HH(x,y)}{2} \le \sinh K d_\HH(x,x').\]

  \end{lemma}

  \begin{proof}
    
    Since $d_\HH(f(x),q) \ge d_\HH(f(y),q)$ and $f$ stretches distances
    along $\lambda$ by $K$, we have
    \begin{align} \label{eq2}
       d_\HH(f(x),q) \ge \frac{d_\HH(f(x),f(y))}{2} = \frac{K d_\HH(x,y)}{2}.
    \end{align}
    Also it suffices to assume that
    \[d_\HH(f(x),f(x'))=d_\HH(f(x),f(\lambda'))\] and hence that the four
    points $f(x), q, q'$ and $f(x')$ form a Lambert quadrilateral (a
    quadrilateral with three right angles). We then have \begin{align}
      \label{eq1}
      \sinh d_\HH(q,q') \cosh d_\HH(f(x),q) = \sinh  d_\HH(f(x),f(x')). 
    \end{align}
    Using the fact that $f$ is $K$--Lipschitz, we obtain 
    \[  d_\HH(f(x),f(x')) \le K d_\HH(x,x'). \]
    The result now follows by plugging (\ref{eq2}) and the above into the
    left and right sides of $(\ref{eq1})$. \qedhere
    
    \end{proof}

  \begin{lemma} \label{Lem:Collar1}

    Let $\gamma$ and $\lambda$ be a pair of intersecting geodesics in
    $\HH$. Given $\ep > 0$, let $U(\gamma)$ be the $w_\ep$--regular
    neighborhood of $\gamma$ and let $L$ be the length of the arc $\lambda
    \cap U(\gamma)$. Let $\phi$ be a hyperbolic isometry with axis $\gamma$
    and translation length $\ep$. Let then $\lambda' = \phi(\lambda)$ and
    $v = d_\HH(\lambda,\lambda')$. Then we have
  
    \[ \sinh (L/2) = \frac{1}{\sinh (v/2)}. \] 

  \end{lemma}

  \begin{proof}

    We will use the Poincar\'e disc model for $\HH$. See Figure
    \ref{fig:Collar} for this proof. Let $p$ be the point of intersection
    between $\gamma$ and $\lambda$, and let $p' = \phi(p)$. Let $q \in
    \lambda$ and $q' \in \lambda'$ be the points such that the arc $[q,q']$
    in $\HH$ is perpendicular to $\lambda$ and $\lambda'$. Without a loss
    of generality, assume the midpoint of $[q,q']$ is the origin $o$ of
    $\HH$. Any isometry taking $\lambda$ to $\lambda'$ must have axis
    passing through $o$ and thus $o$ is also the midpoint of $[p,p']$. Let
    $x$ be the point of intersection between $\lambda$ and a boundary
    component of $U(\gamma)$. Rotation by an angle of $\pi$ about $p$
    preserves $U(\gamma)$ and $\lambda$, and sends $x$ to the other
    intersection of $\lambda$ and $\partial U(\gamma)$. We than have that
    $d_\HH(x,p)= L/2$, since Let $y \in \gamma$ be the foot of the
    perpendicular from $x$ to $\gamma$. The triangles $\triangle(opq)$ and
    $\triangle(xpy)$ are right triangles with hypotenuse $[o,p]$ and
    $[x,p]$ respectively, and $\angle opq = \angle xpy$. Thus, by
    hyperbolic trigonometry of right triangles, we have
    \[ \frac {\sinh d_\HH(o,p)}{\sinh
    d_\HH(o,q)} = \frac{\sinh d_\HH(x,p)}{\sinh d_\HH(x,y)}. \] The formula
    follows since 
    \[ d_\HH(o,p) = \ep/2, \quad d_\HH(o,q) = v/2, \quad d_\HH(x,y) =
    w_\ep, \quad d_\HH(x,p) = L/2. \qedhere \]

  \end{proof}

  \begin{figure}[htp]
    \begin{center}
      \includegraphics{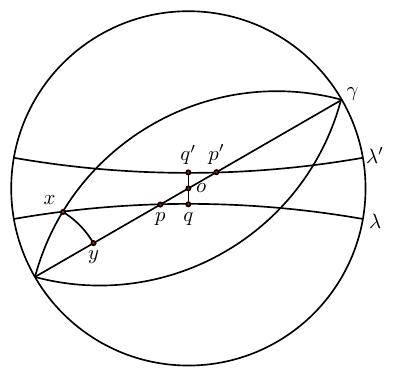}
    \end{center}
    \caption{Figure for the proof of Lemma \ref{Lem:Collar1}.}\label{fig:Collar}

  \end{figure}

  \begin{proposition}\label{Prop:CollarGrowth}
    
    There exists $\ep_0 >0$ such that the following statement holds. Let
    $\g : I \to \T(\s)$ be a Thurston geodesic. Suppose a curve $\gamma$
    intersects $\lambda_\g$ and $\ell_s(\gamma) \leq \ep_0$ for some $s\in
    I$. Then for all $t>s$ with $\ell_t(\gamma) \ge 1$, we have:
    \begin{align*}\label{eq:collargrowth}
      \calL_t \geq \frac{1}{2} e^{t-s}\calL_s, 
    \end{align*}
    where $\calL_t$ is the $X_t$--length of longest arc of $\lambda_\g$ in the
    standard collar of $\gamma$. 

  \end{proposition}
 
  \begin{proof}
  
    Given $\ep$, recall that $w_\ep$ denotes the width of the standard
    collar neighborhood of a curve of length $\ep$ and  $b_\ep$  the length
    of the boundary of the standard collar neighborhood.  Let  $b =
    \sup_{\ep \in (0,1]} b_\ep = b_1$. Explicitly, by (\ref{eq:len bd
    collar}) we have 
    \[ b= \coth \frac{1}{2}.\] 
    Let $\ep_0 \in(0, 1)$ be sufficiently small so that for any positive
    $\ep \le\ep_0$, we have $w_\ep > 2b$. Such $\ep_0$ exists since the
    function $w_\ep$ is decreasing and goes to $+\infty$ at $0$.

    For any $s \in I$, let $\tgamma_s$ be a lift of $\gamma$ to $\tX_s$.
    Let $\phi_s$ be the hyperbolic isometry with axis $\tgamma_s$ and
    translation length $\ell_s(\gamma)$. 
    
    Now fix $s$ and set $\ep = \ell_s(\gamma)$ for some $\ep\in (0,
    \ep_0)$. Let $U_s$ be the $w_\ep$--regular neighborhood of $\tgamma_s$.
    Let $\lambda$ be a lift of a leaf of $\lambda_\g$ that crosses
    $\tgamma_s$ such that $\calL_s$ is the length of $\lambda \cap U_s$.
    Note that $\calL_s \ge 2w_\ep > 4b$. Let $x,y$ be the endpoints of
    $\lambda \cap U_s$. Set $\lambda' = \phi_s(\lambda)$.  

    Suppose $\ell_t(\gamma) = \eta \ge 1$ for some $t > s$. Let $\tf: \tX_s
    \to \tX_t$ be the lift of an optimal map from $X_s \to X_t$. By
    composing with M\"obius transformation if necessary, we may assume
    $\tf\phi_s = \phi_t\tf$. Let $U$ be the $w_\eta$--regular neighborhood
    of $\tgamma_t$ and let $\calL$ be the length of $\tf(\lambda) \cap U$.
    Note that $\calL_t \ge \calL$, so it's enough to prove the statement of
    the proposition for $\calL$. Let $q \in \tf(\lambda)$ and $q'\in
    \tf(\lambda')$ be the endpoints of the common perpendicular of
    $\tf(\lambda)$ and $\tf(\lambda')$. By Lemma \ref{Lem:Collar1} and
    monotonicity of $\sinh$, 
    \[ \sinh \frac{\calL}{2} \ge \frac{1}{\sinh (d_\HH(q,q'))}. \] We now
    proceed to give an upper bound for $\sinh (d_\HH(q,q'))$. 
    
    Without the loss of generality, we can assume that $\dH(\tf(x),q)\geq
    \dH(\tf(y),q)$. Let $x' = \phi_s(x)$. Since $\ell_s(\gamma) < 1$, the
    length along the boundary component of $U_s$ from $x$ to $x'$ is 
    less than $b$, so $d_\HH(x,x') \le b$. This, together with Lemma
    \ref{Lem:Collar2}, yields
    \[ \sinh d_\HH(q,q') \le \frac{ \sinh e^{t-s}
    d_\HH(x,x')}{\cosh (e^{t-s} \calL_s/2)} \le \frac{ \sinh e^{t-s}b}{\cosh
    (e^{t-s} \calL_s/2)} \le \frac{1}{e^{e^{t-s} \left( \frac{\calL_s}{2} - b
    \right)}}\]
    Since $\calL_s > 4b$ and $\sinh^{-1}x\geq \log{x}$ for $x\geq 1$, we obtain
    \[ \calL \ge 2 \sinh^{-1} e^{e^{t-s} \left( \frac{\calL_s}{2} - b
    \right)} \ge 2e^{t-s} \left(  \frac{\calL_s}{2} - b \right) \ge
    \frac{1}{2}\calL_s e^{t-s}. \qedhere \]

  \end{proof}  

\section{A notion of being horizontal}

  Let $\g: I\to\T(S)$ be a Thurston geodesic and let $\lambda_\g$ be the
  maximally stretched lamination. The main goal of this section is to
  develop a notion of a curve being horizontal along  $\g$ in such a way
  that as soon as $\alpha$ is sufficiently horizontal, it remains
  horizontal and its length grows essentially exponentially. There are two
  stages, weakly horizontal and strongly horizontal. 
  
  This latter notion requires a rather technical definition, but roughly
  speaking, it means that curve $\alpha$ fellow travels $\lambda_\g$ both
  geometrically and topologically for a long time. We will prove that
  strongly horizontal curves stay strongly horizontal and grow
  exponentially in length. Then, a curve $\alpha$ is weakly horizontal if
  it stays parallel with a leaf of $\lambda_\g$ through a long collar
  neighborhood of a very short curve. We will show that a weakly horizontal
  curve quickly becomes strongly horizontal. 

  \subsection{Strongly horizontal}

  We will start with the definition of strongly horizontal. This follows
  closely to the definition of horizontal as introduced in \cite{LRT2}, but
  with an adjustment of the constants involved. The constant $\ep_h \ge
  \ep_B$ in the following definition will be determined in Lemma \ref{Lem:Constants}. 
   
  \begin{definition}[Strongly Horizontal] \label{Def:Shor}
    Let $\g \from I \to \T(S)$ be a Thurston geodesic and denote $X_t =
    \g(t)$. Given $n \in \mathbb N$ and $L>0$, we say a curve $\alpha$ is
    {\em $(n,L)$--horizontal} on $X=X_t$ if there exists an $\ep_h$--short
    curve $\gamma$ on $X$ and a leaf $\lambda$ in $\lambda_\g$ such that
    the following statements hold:  
    \begin{itemize}
      \item[(H1)] In the universal cover $\widetilde X \cong \HH$, there is
        a collection $\{\tgamma_1, \ldots \tgamma_n\}$ of lifts of the
        curve $\gamma$ and a lift $\tlambda$ of $\lambda$ intersecting each
        $\tgamma_i$ at a point $p_i$ on $\tl$ (the $p_i$'s are indexed by
        the order of their appearances along $\tlambda$) such that
        $d_\HH(p_i, p_{i+1}) \ge L$ for $i=1, \ldots, n-1$.
      \item[(H2)] There is a lift $\talpha$ of $\alpha$ such that $\talpha$
        intersects $\tgamma_i$ at a point $q_i\in \HH$ with $\dH(p_i, q_i)
        \le \ep_h$ for each $i=1,\ldots, n$.
    \end{itemize}
    We will call $\gamma$ as above an {\em anchor curve} for $\alpha$, and
    $\talpha$ an {\em$(n,L)$--horizontal lift} of $\alpha$. Moreover, we
    call the segment $[q_1,q_n] \subseteq\talpha$ the \emph{horizontal
    segment} of $\talpha$, and the projection of $[q_1,q_n]$ to $X$ the
    horizontal segment of $\alpha$. 
  \end{definition} 
   
  We recall the following proposition from \cite{LRT2}. 

  \begin{proposition}{\cite[Proposition 4.6]{LRT2}} 
    \label{prop : enhance to horizontal} 
    There is a constant  $L_0>0$ such that the following statement holds.
    Let $X$ be a hyperbolic metric on $S$, $\gamma$ an $\ep_B$--short curve
    on $X$, $\lambda$ is a complete simple geodesic on $X$. Then for any
    $n\in \mathbb N$ and $L\geq L_0$, if there are $n$ lifts
    $\{\tgamma_i\}$ of $\gamma$ and a lift $\tlambda$ of $\lambda$ that
    satisfying condition (H1) of Definition \ref{Def:Shor}, and $\alpha$ is
    a curve on $X$ with a lift $\talpha$ that lies $\ep_B$--close to the
    segment $[p_1,p_n]$ in $\tlambda$, then for all $i=3,\ldots n-2$,
    $\talpha$ intersects $\tgamma_i$ at a point $q_i$ with
    $d_\HH(p_i,q_i)\leq \ep_B$.

  \end{proposition}
  
  Recall that we have fixed once and for all the Bers constant $\ep_B$,
  which fixes the associated collar width $w_{\ep_B}$ and the length
  $b_{\ep_B}$ of the boundary of the collar neighborhood of a curve of
  length $\ep_B$. In the following lemma we set several constants to
  quantify the notions of horizontal curve that we require in the rest of
  the paper.

  \begin{lemma} [Constants] \label{Lem:Constants} 

    Let $A$ be the constant of Lemma \ref{Lem:ConstantA} and $M$ the
    constant from Lemma \ref{Lem:ConstantM}. Then there are positive
    constants $\ep_w$, $\ep_h$, $L_0$ and $n_0\in \mathbb N$ such that the
    following conditions are satisfied simultaneously.
    \begin{enumerate}
      \item $ \ep_w< \min \set{1, 2w_{\ep_B}}$ and satisfies
        \propref{CollarGrowth},\label{ew} 
      \item $\ep_h=\ep_B+4w_{\ep_w}+2M$\label{eh}
      \item $ L_0 = 10 \ep_h$.\label{L01}
      \item $L_0$ satisfies Proposition \ref{prop : enhance to
      horizontal}.
    \label{L02} \item $L_0\geq w_{\ep_B}+6b_{\ep_B}$
        \label{L03} 
      \item $\frac{1}{\ep_w}w_{\ep_w}\geq n_0\lceil L_0+1\rceil+5+A$.
        \label{L04} 
      \item $\sinh(1+w_{\ep_B})\sinh\frac{\ep_B}{2}\leq
        \exp(\frac{1}{6}L_0-1)$. \label{L05}
      \item $n_0\geq 8$.\label{L06}
    \end{enumerate}
  \end{lemma}
    
  \begin{proof}
    
    First note that for all $\ep_w$ sufficiently small (\ref{ew}) holds.
    Then we may define $\ep_h$ and $L_0$ as in (\ref{eh}) and (\ref{L01}).
    Since $L_0$ gets arbitrary large by choice of $\ep_w$ small enough,
    (\ref{L02}) and (\ref{L03}) hold for $\ep_w$ small enough. Also
    (\ref{L04}) holds for $\ep_w$ small enough, because the right hand side
    of the inequality in (\ref{L04}) is a linear function of $w_{\ep_w}$
    and  $\frac{1}{\ep_w}w_{\ep_w}$ is larger than that for all $\ep_w$
    small enough. Finally (\ref{L05}) holds because the right hand side of
    (\ref{L05}) could be made arbitrary large for $\ep_w$ small enough.
    \qedhere
  
  \end{proof}

  From now on we fix constants $n_0, \ep_w, \ep_h$ and $L_0$ that satisfy
  the conditions of \lemref{Constants}. The following theorem provides us
  with a control on the growth of the quantifiers of a strongly-horizontal
  curve moving forward along a Thurston geodesic. It is a generalization of
  Theorem 4.2 in \cite[\S4]{LRT2} to the situation that the curve is inside
  a subsurface and the maximal stretch lamination intersects the
  subsurface.
  
  \begin{theorem} \label{Thm:Horizontal}
    There exists $s_0\geq 0$  such that the following statement holds. Let
    $\g:I\to \T(S)$ be a Thurston geodesic, and let $R\subseteq S$ be a
    non-annular subsurface that  intersects $\lambda_\g$ essentially.
    Suppose that a curve $\alpha \in \calC(R)$ is $(n_s,L_s)$--horizontal
    at $s \in I$ where $n_s\geq n_0$ and $L_s \geq L_0$. Then, 
    \begin{enumerate}
      \item For any $t\geq s+s_0$, the curve $\alpha$ is
        $(n_t,L_t)$--horizontal on $X_t$ where 
        \[ n_t \gmul n_s \quad \text{and} \quad L_t\ge L_s .\]
      \item Moreover, if for some $d>1$ we have $d_R(X_s,X_t)\geq d$, then 
        \[ \log \frac{n_t}{n_s} \gmul d  \quad\text{and}\quad L_t n_t \gmul
        e^{t-s} L_s n_s. \] 
    \end{enumerate}
  \end{theorem}

  \begin{proof}
  
    The proof of both statements is very similar to that of  \cite[Theorem
    4.2]{LRT2}, so we only sketch it here. 
  
    Let $\talpha, \tlambda, p_i, q_i$, $\gamma$ and $\tgamma_i$ be as in
    \defref{Shor}. Let $f:X_s\to X_t$ be an $e^{t-s}-$ Lipschitz map. Let
    $\tf:\tX_s\to \tX_t$ be a lift of $f$ that preserves $\tlambda$ and
    that fixes pointwise its endpoints. The first step is to show that  the
    geodesic representative  of $\tf(\talpha)$ is distance at most $1$ to
    $\tlambda$ from $\tf(p_3)$ to $\tf(p_{n_s-2})$.
   
    We will apply Proposition 3.5 in \cite[\S4]{LRT2}, and to do so we need
    that, if exists, the closest to $\talpha$ point $c$ on $\tlambda$ be at
    least $2\ep_h$ away from both $p_1$ and $p_{n_s}$. If this is not the
    case, note that at most one of the points $p_i$ is within $2\ep_h$ of
    $c$, so we lose at most one point. Depending on that, label $l=1$ or
    $2$ and $r=n_s$ or $n_s-1$.
  
    By Proposition 3.5 in \cite[\S4]{LRT2} there are  translates
    $\tlambda_l$ and $\tlambda_r$ of $\tlambda$ by hyperbolic isometries
    with axes $\tgamma_l$  and  $\tgamma_r$  so that the endpoints
    $\alpha_l$ and $\alpha_r$ of $\talpha$ are sandwiched between those of
    $\tlambda$ and $\tlambda_l$  and of $\tlambda$ and $\tlambda_r$
    respectively, and so that $p_l$ and $p_r$ are at most $4\ep_h+3$ from
    respectively $\tlambda_l$ and $\tlambda_r$.  
  
    We now apply $\tf$. Let $\talpha'$ be the geodesic representative  of
    $\tf(\talpha)$. Its endpoint $\tf(\alpha_l)$ is between $\tlambda$ and
    $\tf(\tlambda_l)$, and   $\tf(\alpha_r)$ is between $\tlambda$ and
    $\tf(\tlambda_r)$. Since $\tf(\lambda_l)$ is within distance
    $e^{t-s}(4\ep_h+3)$ of $\tf(p_l)$, and $\tf(\lambda_r)$ is within
    distance $e^{t-s}(4\ep_h+3)$ of $\tf(p_r)$, we have that $\talpha'$
    stays  within distance $e^{t-s}(4\ep_h+3)$ of $\tlambda $ from
    $\tf(p_l)$ to $\tf(p_r)$.  It now follows from \lemref{closer} that
    $\talpha'$ is $1$-- close to  $\tlambda $ from  $\tf(p_{l+1})$ to
    $\tf(p_{r-1})$. 
  
    The next step is to show that any segment $[\tf(p_i),\tf(p_{i+3})]$ for
    $i=l, \cdots, r-3$, intersects a lift of an $\ep_B$-- short curve. Let
    $\tomega$ be any such segment and let $\omega$ be its projection by
    $\pi:\tX_t\to X_t$.  
   
    If $\pi: \tomega \to \omega$ is not injective, then $\tlambda$ is a
    lift of a closed curve $\lambda$ which on $X_s$ has to satisfy
    $\ell_{s}(\lambda)\geq w_{\ep_h}$ and hence  $\ell_{t}(\lambda)\geq
    e^{s_0}w_{\ep_h}$. Assuming that $s_0>\log\frac{\ep_B}{w_{\ep_h}}$
    guarantees that $\lambda$ intersects an $\ep_B$--short curve, and hence
    that $\tomega$ intersects a lift of the $\ep_B$-short curve.
   
    If $\pi: \tomega \to \omega$ is injective, then we prove that it
    intersects an $\ep_B$ --short curve as in Claim 4.8 of
    \cite[\S4]{LRT2}, except that instead of \cite[Lemma 3.2]{LRT2} we
    apply \lemref{Tooshort}, and we require that $s_0\geq
    \log\frac{D_0}{w_{\ep_h}}$ where $D_0=D_0(\ep_h)$ is the constant from
    \lemref{Tooshort}, and check that $L_0> 6\ep_h$ by condition
    (\ref{L01}) of \lemref{Constants}. 
  
    Hence for some $\ep_B$--short curve (call it $\beta$) and some
    $n_t\gmul n_s$, the segment $[\tf(p_l),\tf(p_r)]$ intersects at least
    $n_t$ lifts of $\beta$, such that any two of the intersection points
    are at least $L_t=L_se^{t-s}$ apart. Applying  Proposition \ref{prop :
    enhance to horizontal} we conclude $\alpha$ is $(n_t, L_t)$--horizontal
    which proves  part (1) of the theorem.
   
    We now sketch the proof of  part (2).  Suppose that $d_R(X_s, X_t)\geq
    d>1$. Let $\calP$ be an $\ep_B$--short pants decomposition of $X_t$.
    Note that for small values of $d$ part (2) follows from part (1), hence
    we assume that $d$ is large, which implies in particular that $\gamma$,
    the anchor curve of $\alpha$ at $X_s$ as chosen as in \defref{Shor},
    intersects every curve in $\calP$ that enters $R$. Let 
    \[
    m=\underset{\beta\in \calP}{\min}\Big\{\I_R(\beta,\gamma)\st \beta\cap
    R\neq \emptyset\Big\}.
    \]
    From \eqnref{di} we have 
    \begin{equation}\label{Eqn:A-m}
      d\lmul \log{m},
    \end{equation}
    also note that $m$ satisfies  
    \begin{equation}\label{Eqn:m}
     m\leq \frac{\ell_t(\gamma)}{w_{\ep_B}}\leq \frac{e^{t-s}\ep_h}{w_{\ep_B}}.
   \end{equation}
    
   Cut the segment $[\tf(p_{l+1},\tf(p_{r-1})]$ into
   $\max{\{\lfloor\frac{m(r-l-2)w_{\ep_B}}{\ep_h}\rfloor, 1\}}$ pieces of
   equal length. Let $\tomega$ be one of the pieces and let $\omega$ be its
   projection to $X_t$.  We will show that $\omega$ intersects a curve in
   $\calP$. 
   
   If $\pi: \tomega \to \omega$ is not injective, then $\omega$ is a closed
   leaf of $\lambda$. Since $s_0> \log{\frac{\ep_B}{w_{\ep_h}}}$,
   $\ell_t(\omega) > \ep_B$ and therefore $\omega$ intersects some curve in
   $\calP$. 
   
   Suppose now that $\pi: \tomega \to \omega$ is injective. Since the
   segment $[\tf(p_{l+1}),\tf(p_{r-1})]$ has length at least
   $L_0(r-l-2)e^{t-s}$ we have
   \begin{equation}\label{Eqn:omega}
   \ell_t(\omega)\geq \frac{e^{t-s}\ep_h}{m}\frac{L_s}{w_{\ep_B}}.
   \end{equation}
  
   Assume now for contradiction that $\omega$ misses all the curves in
   $\calP$. Then it is contained in a pair of pants $P$ with $\ep_B$--short
   boundary components.

   By \cite[Lemma 3.1]{LRT2} there is a boundary component $\beta$ of $P$
   such that an arc $\tau$ of $\omega$ is contained in $U(\beta)$ and
   $\ell_t(\tau)\geq \frac{1}{3}\ell_t(\omega)$.  
   \begin{claim}\label{intersects}
     The curve $\beta$ intersects $R$.
   \end{claim}

   \begin{proof}   

     Since the curve $\alpha$ is contained in $R$, it suffices to show that
     $\alpha$ intersects $\beta$ or that $\alpha$ and $\beta$ are the same
     curve. In fact it suffices to show that there is a point on $\alpha$
     that is distance less than $w_{\ell_t(\beta)}$ from $\beta$, that is,
     that $\alpha$ enters the standard collar of $\beta$ at $X_t$. Recall
     that
     $\talpha'$ is 1--close to $\tlambda$ from $\tf(p_{l+1})$ to
     $\tf(p_{r-1})$. Let $\tilde{\tau}\subseteq\tomega$ be the lift of
     $\tau$, and $\tbeta$ the  appropriate lift of $\beta$. Let
     $[a_1,a_2]\subseteq\talpha'$ be the longest segment in the
     $1$--neighborhood of $\tilde{\tau}$. Then 
     \begin{center} 
     $d_\HH(a_1,a_2)\geq \frac{1}{3}L_0-2$
     and $d_\HH(a_i,\tilde{\beta})\leq 1+w_{\ell_t(\beta)}.$
     \end{center}
     Now by \lemref{closer}  applied to $\ep_1=w_{\ell_t(\beta)}$ and
     $\ep_2=1+w_{\ell_t(\beta)}$, the midpoint $p$ of $[a_1,a_2]$ is
     distance less than $w_{\ell_t(\beta)}$ away from $\beta$ as long as 
     \begin{equation}\label{distanceL} 
     d_\HH(a_1,a_2)\geq 2\log 2 \frac{\sinh(1+w_{\ell_t(\beta)})}{\sinh(w_{\ell_t(\beta)})}.
     \end{equation}

     Since the function $\log 2 \frac{\sinh(1+x)}{\sinh(x)}$ is decreasing
     and $w_{\ell_t(\beta)}\geq w_{\ep_B}$, we have, by the choice of $L_0$
     in (\ref{L05}) of \lemref{Constants},  that the inequality
     \ref{distanceL} holds.

     Hence $p$ is within distance $w_{\ell_t(\beta)}$ of $\tilde\beta$,
     which implies that $\alpha$ intersects or coincides with $\beta$, and
     so $\beta$ intersects $R$. \qedhere 

   \end{proof}

   The claim and the definition of $m$ imply that $\gamma$ intersects
   $\beta$ at least $m$ times. Let $\gamma'=f(\gamma)$ be the image of the
   geodesic $\gamma$. Then $\gamma'$ also intersects $\beta$ at least $m$
   times, and there are at least $m$ disjoint arcs in $\gamma'$ that
   connect both  boundary components of $U(\beta)$. Let $\eta$ be a
   shortest such arc. Clearly    $\ell_t(\eta)\leq \frac{e^{t-s}\ep_h}{m}.$

   If $\I(\tau,\eta)\geq 2$, the distance between two consecutive
   intersections along $\tau$ is at least $D=e^{t-s}w_{\ep_h}$. By
   requiring that $s_0> \log (6 \frac{b_{\ep_B}}{w_{\ep_h}})$, we
   guarantee $D> 6b_{\ep_B}$. Now by  \lemref{len comp in collar} 
   \begin{equation}\label{Eq:tau1}
     \ell_t(\tau)\leq \frac{\ell_t(\eta)}{1-\frac{3b_{\ep_B}}{D}}<
     \frac{2\ep_he^{t-s}}{m}
   \end{equation} 
   which implies
   \[
   \ell_t(\omega)< \frac{6\ep_he^{t-s}}{m}.
   \]
   When $\tau$ and $\eta$ intersect at most once, then by the second
   inequality of Lemma \ref{Lem:len comp in collar}, 
   \[\ell_t(\tau)\leq  \ell_t(\eta)+2b_{\ep_B}\]
   So we obtain $$\ell_t(\omega)\leq \frac{3\ep_he^{t-s}}{m}+6b_{\ep_B}\leq
   \frac{\ep_he^{t-s}}{m}\left(\frac{3w_{\ep_B}+6b_{\ep_B}}{w_{\ep_B}}\right)$$
   By \lemref{Constants} both upper bounds on $\ell_t(\omega)$ contradict
   the \eqnref{omega}. Hence $\omega$ intersects a curve in $\calP$.  
    
   Recall that we have cut $[\tf(p_{l+1},\tf(p_{r-1}))]$ into $\max{\left\{
     \left\lfloor\frac{m(r-l-2)w_{\ep_B}}{\ep_h}  \right\rfloor, 1
     \right\}}$ segments, and that the length of each segment is at least
   $L_s$.   We also showed that each segment intersects a geodesic that
   projects to some $\ep_B$ --short curve. By choosing every other segment
   if necessary we guarantee that  the distance between the intersection
   points is at least $L_s$. Then for some $\beta\in \calP$ the number
   $n_t$ of segments that intersect a lift of $\beta$ satisfies $n_t\gmul
   mn_s$. By \eqnref{A-m}, $\log m\gmul d$, thus applying Proposition
   \ref{prop : enhance to horizontal} finishes the proof of the second part
   of the theorem. \qedhere
    
  \end{proof}
  
  The following proposition gives some fundamental examples of horizontal
  curves on a hyperbolic metric $X$ on $S$: a curve of length at least $1$
  with $\lambda_\g$ twisting a lot about it, and a curve that together with
  a leaf of $\lambda_\g$ twists a lot about some curve of length at least
  $1$.
  
  \begin{proposition} \label{Prop:Horizontal1}

    Let $\g:I\to \T(S)$ be a Thurston geodesic. Let $\gamma$ be a curve,
    $\tau$ be the shortest curve on $X_t$ intersecting $\gamma$, and
    $\lambda$ be a leaf of $\lambda_\g$. Suppose $\ell_t(\gamma) \geq  1$
    and $d_\gamma(\tau,\lambda) \ge  n \left\lceil
    \frac{L_0}{\ell_t(\gamma)}+1 \right\rceil+5$ for some $n \in
    \mathbb{N}$. Then $\gamma$ and any curve $\alpha$ with $d_\gamma(\alpha,\lambda) \le
    3$ is $(n,L_0)$--horizontal on $X_t$ with anchor curve $\tau$.
  
  \end{proposition}

  \begin{proof} 

    If $\ell_X(\gamma)>\ep_B$, then it intersects an $\ep_B$--short curve
    by definition of $\ep_B$. If $\gamma$ itself is $\ep_B$--short,
    applying Lemma \ref{Lem:ConstantM} we have that the curve $\tau$
    satisfies $\ell_t(\tau)\leq 4w_{1}+2M$. In both cases by the choice of
    $\ep_h$ in (\ref{eh}) of Lemma \ref{Lem:Constants} we have
    $\ell_t(\tau)\leq \ep_h$.
 
    Denote the standard collar neighborhood of $\gamma$ by $U$, and choose
    arcs of $\tau,\lambda$ and $\alpha$ inside $U$, denoting them
    $\tau_\gamma$, $\lambda_\gamma$ and $\alpha_\gamma$ respectively. Let
    $\eta$ be an arc  perpendicular to $\gamma$ crossing $U$ and disjoint
    from $\tau_\gamma$. Such an arc exists since $\tau$ is chosen to be the
    shortest possible curve that intersects $\gamma$.
  
    Now the assumption on the annular coefficients and the triangle
    inequality imply $$d_\gamma(\tau,\alpha)>  n\Big\lceil
    L_0/\ell_t(\gamma)+1\Big\rceil,$$ thus $\tau_\gamma$ intersects both
    $\alpha_\gamma$ and $\lambda_\gamma$ at least $ n\Big\lceil
    L_0/\ell_t(\gamma)+1\Big\rceil$ times. 
  
    Choosing a lift $\tgamma$ of $\gamma$ to $\tilde X_t=\HH$, let
    $\widetilde{U}$ be the standard collar neighborhood of $\tgamma$.
    $\widetilde{U}$  contains infinitely many lifts of $\tau_\gamma$ and
    $\eta$. We can choose lifts $\talpha$ and $\tlambda$ of $\alpha_\gamma$
    and $\lambda_\gamma$ in $\widetilde U$ so that they  intersect together
    the same $n\lceil L_0/\ell_t(\gamma)+1\rceil$ lifts of $\tau_\gamma$.
    The length of $\tau$ is at most $\ep_h$, so the distance between the
    intersection points of $\tlambda$ and $\talpha$ with a lift of
    $\tau_\gamma$ is bounded by $\ep_h$. 
 
    Since the distance between any two consecutive lifts of $\eta$ is
    $\ell_t(\gamma)$ and any lift of $\tau_\gamma$ is strictly sandwiched
    between two consecutive lifts of $\eta$, taking every $\lceil
    \frac{L_0}{\ell_t(\gamma)}+1\rceil$ lift of $\tau_\gamma$ insures that
    the distance between intersections of $\tlambda$ with these lifts is at
    least  $$\left \lceil \frac{L_0}{\ell_t(\gamma)} \right\rceil
    \ell_t(\gamma)\geq L_0.$$   Then $\talpha$ also intersects these $n$
    lifts of $\tau_\gamma$ and we can conclude that $\alpha$ is
    $(n,L_0)$--horizontal. As for $\gamma$, it is   $(n,L_0)$--horizontal
    since its lift $\tgamma$ intersects all the lifts of $\tau_\gamma$ that
    $\tlambda$ intersects and the distance between the intersection points
    is also bounded by $\ep_h$. \qedhere
 
  \end{proof}

  In the following, we will show that if a curve is $(n,L)$--horizontal at
  a point on a Thurston geodesic then its length essentially grows
  exponentially along
  the geodesic. Let us now recall the notion of horizontal and vertical
  components of a simple closed curve that intersects a leaf of
  $\lambda_\g$, that were introduced in \cite{coarseandfine}.
  \begin{definition} Let $X\in\T(S)$ and let $\alpha$ be a simple closed
    geodesic on $X$. If $\alpha$ intersects a  leaf $\lambda$ of
    $\lambda_\g$, define $V$ to be a shortest arc with endpoints on
    $\lambda$ that, together with an arc $H$ of $\lambda$, form a curve
    homotopic to $\alpha$.  Thus $V$ and $H$ meet orthogonally and $\alpha$
    passes through the midpoints of both of these arcs.  If $\alpha$ is a
    leaf of $\lambda_\g$, then we set $H = \alpha$ and let $V$ be the empty
    set.

    Define $h_X(\alpha)$ and $v_X(\alpha)$ to be the lengths of $H$ and $V$
    respectively. 
  \end{definition}

  \begin{proposition}\label{Prop:Growth}

    Suppose $\alpha$ is $(n_s,L_s)$--horizontal for $s \in I$ with $L_s \geq
    L_0$, and $n_s\geq n_0$. Then for any $t \ge s$, \[
      \ell_t(\alpha) \ge e^{t-s}\min
      \left\{\frac{1}{2}\ell_s(\alpha),\frac{L_s}{10}  \right\}.\]
     In particular,  
      \[\ell_t(\alpha) \ge w_{\ep_h}e^{t-s}.\]

  \end{proposition}
  
  \begin{proof}

    Let $\alpha$ be an $(n_s,L_s)$--horizontal curve at $X_s$ and
    $\talpha_s$ a horizontal lift of $\alpha$ to $\tX_s \cong \HH$. That
    is, there is a lift $\tl$ of a leaf of $\lambda_\g$, $n_s$--lifts $\{
      \tgamma_i \}$ of an $\ep_h$--short curve $\gamma$, such that
    $\dH(p_i, q_i) \le \ep_h$ and $\dH(p_i, p_{i+1}) \ge L_s$, where $p_i$
    and $q_i$ are respectively the intersection of $\tgamma_i$ with
    $\tlambda$ and $\talpha_s$. 
    
    Let $t \ge s$, and let $f _t\from X_s \to X_t$ be an
    $e^{t-s}$--Lipschitz map, and $\tf_t \from \tX_s \to \tX_t$ be a lift
    of $f_t$. Recall that $\tf_t$ is chosen to preserve $\tlambda$, and it
    stretches it by a factor of $e^{t-s}$. Let $\phi_s$ be the hyperbolic
    isometry preserving $\talpha_s$ with translation length
    $\ell_s(\alpha)$, oriented in the direction of
    $\overrightarrow{q_1q_{n_s}}$, and let $\phi_t$ be be such that $\phi_t
    \tf _t= \tf_t \phi_s$. That is, the axis of $\phi_t$ is the geodesic
    representative $\talpha_t$ of $\tf(\talpha_s)$. In the proof of
    \thmref{Horizontal} it was shown that $\talpha_t$ stays $\ep_h$--close
    to $\tlambda$ from $\tf_t(p_{l})$ to $\tf_t(p_{r})$ where $r-l\geq
    n_s-3$. Let $[b_t,b'_t]\subset \tlambda$ be the largest segment with
    $d_\HH(b_t,\talpha_t)= d_\HH(b_t',\talpha_t)=\ep_h$.  Denote $a_t$  and
    $a'_t$ the feet of the perpendiculars from $b_t$ and $b'_t$ to
    $\talpha_t$. We have
    \begin{equation}\label{longsegment}
      d_\HH(a_t,a'_t)\geq d_\HH(b_t,b'_t)-2\ep_h\geq
      e^{t-s}L_s(n_s-4)-2\ep_h\geq e^{t-s}L_s(n_s-5)
    \end{equation}
    
    Let $k_t$ be the largest number such that $\phi^{k_t}_t(a_t)\in
    [a_t,a'_t]$. That is, $k_t=\Big\lfloor \frac{d_\HH(a_t,a'_t)}{\ell_t(\alpha)}
    \Big\rfloor$. There are several cases to consider. 
    
    First suppose that  $k_t<5$.  Then by \ref{longsegment},
   
    $$5\ell_t(\alpha)\geq d_\HH(a_t,a'_t)\geq e^{t-s}L_s(n_s-5)$$ which
    implies 
     \begin{equation}\label{alphalong}
     \ell_t(\alpha)\geq \frac{L_se^{t-s}}{5}.
    \end{equation}
     
    Now suppose that $k_t\geq 5$.  Then by \lemref{Intersect}, either
    $\lambda$ and $\alpha$ intersect, or $\alpha$ is a closed leaf of
    $\lambda_\g$. In the later case the length of $\alpha$ grows
    exponentially. In the former case, we will estimate the horizontal and
    vertical components of $\alpha$ with respect to the leaf $\lambda$ on
    $X_t$. We will write $h_t$ and $v_t$ instead of $h_{X_t}(\alpha)$ and
    $v_{X_t}(\alpha)$. 
    
    By a basic hyperbolic geometry computation,
    \begin{equation}\label{ratio}
      \frac{\sinh(v_s/2)}{\sinh(\ell_s(\alpha)/2)} \leq
      \frac{\sinh(\ep_h)}{\sinh(d_\HH(a_s,a'_s)/2)}\leq 2e^{\ep_h-\frac 1 2
      d_\HH(a_s,a'_s)}.
    \end{equation}
    
    We have $$  \sinh(v_s/2)\leq 2 \sinh(\ell_s(\alpha)/2)e^{\ep_h-\frac 1
    2 d_\HH(a_s,a'_s)}.$$ If $\ell_s(\alpha)\leq d_\HH(a_s,a'_s)$, that is,
    $k_s\geq 1$, then differentiating when $\ell_s(\alpha)\leq 2$ and when
    $\ell_s(\alpha)>2$ and using   \ref{longsegment} and the fact that
    $L_s\geq L_0=9\ep_h$,  $\ep_h\geq 1$ and $n_s\geq 8$ we get the
    following generous estimate $$v_s \leq \ell_s(\alpha)/8$$ which implies
    \begin{equation}\label{horizontal2}
    h_s(\alpha)-v_s(\alpha)\geq \frac 1 2 \ell_s(\alpha).
    \end{equation}
    By Lemma 3.4 in \cite{coarseandfine}  and definition of $h_t$ we have
    $\ell_t(\alpha)\geq h_t\geq e^{t-s}(h_s-v_s)$ and so together with
    \ref{horizontal2} this implies $$\ell_t(\alpha)\geq
    \frac{1}{2}e^{t-s}\ell_s(\alpha).$$ The last case to consider is when
    $k_t\geq 5$ so $\alpha$ and $\lambda$ intersect, but  $k_s=0$. That is,
    the length of $\alpha$ on $X_s$ is bigger than the length of the
    segment which fellow travels $\lambda$, and on $X_t$ the leaf $\lambda$
    wraps about $\alpha$. 
    
    Since the endpoints of $\talpha_t$ vary continuously with $t$, the
    function $d_\HH(a_t,a'_t)/\ell_t(\alpha)$ is continuous, and therefore
    there is some $s'\in(s,t)$ such that $1\leq k_{s'}< 5$. Then, as in
    \ref{alphalong}, we have $\ell_{s'}(\alpha)\geq\frac 1 5 L_se^{s'-s}.$
    Further,  on $X_{s'}$ the horizontal and vertical components $h_{s'}$ and
    $v_{s'}$ of $\alpha$ are estimated as $h_s$ and $v_s$ above and so
    $$h_{s'}-v_{s'}\geq \frac 1 2 \ell_{s'}(\alpha). $$ Applying again
    Lemma 3.4 of \cite{coarseandfine}  gives $$ \ell_t(\alpha)\geq \frac 1
    2 e^{t-{s'}}\ell_{s'}(\alpha)\geq \frac {L_s}{10}e^{t-s}. $$ Hence we
    have shown that $$\ell_t(\alpha)\geq e^{t-s}\min \left\{\frac 1 2
    \ell_s(\alpha),\frac 1 {10} L_s \right\}.$$ The last inequality in the
    statement follows from the first one and the fact that $\ell_s(\alpha)$
    is at least $2w_{\ep_h}$ since $\alpha$ intersects an $\ep_h$-- short
    curve, and that $w_{\ep_h}\leq \ep_h\leq L_s/10$. \qedhere

  \end{proof}
  
  \subsection{Weakly horizontal}

  We now define the notion of weakly horizontal. 

  \begin{definition}[Weakly Horizontal] \label{Def:Whor}

    Let $\g \from I \to \T(S)$ be a Thurston geodesic. We say a curve
    $\alpha$ is \emph{weakly horizontal} on $X_t=\g(t)$ if there is an
    $\ep_w$--short curve $\gamma$ on $X_t$ intersecting $\alpha$ and
    $\lambda_\g$ such that inside the standard collar $U_X(\gamma)$ of
    $\gamma$, there is an arc of $\alpha \cap U_X(\gamma)$ and an arc of
    $\lambda \cap U_X(\gamma)$ that are disjoint.

  \end{definition}
  
  For convenience, we will also call $\gamma$ an \emph{anchor curve} for
  $\alpha$. Note that when $\alpha$ is weakly horizontal, then
  $d_\gamma(\alpha,\lambda_\g) \le 3$. 

  We also call any $\ep_w$--short curve that intersects $\lambda_\g$ a
  \emph{pre-horizontal} curve. Note that the anchor curve of a
  weakly-horizontal curve is pre-horizontal. 

  The following theorem tells us that a weakly horizontal curve and its
  anchor curve both become strongly horizontal as soon as the length of the
  anchor curve grows.

  \begin{theorem} \label{Thm:Whor}
    
    Let $\g:I\to \T(S)$ be a Thurston geodesic. Suppose $\alpha$ is weakly
    horizontal on $X_s$ with anchor curve $\gamma$. Let $t > s$ be the
    first time that $\ell_t(\gamma)=1$. Then $\alpha$ and $\gamma$ are both
    $(n_0,L_0)$--horizontal on $X_t$. 

  \end{theorem}
  
  \begin{proof}

    Let $\lambda$ be a leaf of $\lambda_\g$ crossing $\gamma$. For any $u
    \in [s,t]$ denote $U_u(\gamma):=U_{X_u}(\gamma)$  the standard collar
    neighborhood of $\gamma$ in $X_u$, and let $\calL_u$ be the
    $X_u$--length of a longest arc of $\lambda_\g \cap U_u(\gamma)$. 
    
    Let $\ep $ be the length $\ell_s(\gamma)$ and recall that by the
    definition of weakly horizontal $\ep\leq \ep_w$.  By Proposition
    \ref{Prop:CollarGrowth}, we have $ \calL_t \geq \frac{1}{2}
    e^{t-s}\calL_s$. Also since $\ell_t(\gamma) = 1$ and $\ell_s(\gamma) =
    \ep$, from  Theorem \ref{Thm:Thurston} (\ref{th:exp}) we have that
    $e^{t-s} \ge \frac{1}{\ep}$.

    Further, the length $\calL_s$ is at least the width of $U_s(\gamma)$,
    so $\calL_s \ge 2w_\ep$. Thus we obtain
    \[ \calL_t \geq \frac{1}{\ep}w_\ep. \]
      Again since $\ell_t(\gamma)=1$, letting $\tau$ be a shortest curve on
      $X_t$ transverse to $\gamma$ from  \lemref{ConstantA} we have that 
    \[
      d_\gamma(\tau,\lambda)\geq \calL_t - A .
    \]
    Putting the above two inequalities together and using the fact that
    $\ep \le \ep_w$ and the property of $\ep_w$ from part (\ref{L04}) of
    Lemma \ref{Lem:Constants}, we obtain 
    \[
      d_\gamma(\tau,\lambda) \ge \frac{1}{\ep}w_\ep-A \ge n_0\Big\lceil L_0+1\Big\rceil+5.
    \]
    The above inequality, the fact that $d_\gamma(\alpha,\lambda) \le 3$
    and \propref{Horizontal1} imply that both $\gamma$ and $\alpha$ are
    $(n_0,L_0)$--horizontal on $X_t$, as desired. \qedhere

  \end{proof}

  \begin{definition}[Horizontal]

    We say a curve $\alpha$ is \emph{horizontal} on $X$ if it is either
    weakly horizontal or $(n_0,L_0)$--horizontal on $X$. Similarly, we will
    say a multicurve $\sigma$ is horizontal if each component of $\sigma$
    is horizontal.

  \end{definition}
  
  The following proposition gives us a uniform upper bound for the diameter
  of the shadow of any Thurston geodesic to the curve graph of a subsurface
  if the boundary of the subsurface is horizontal on the initial point of
  the segment. 
 
  \begin{proposition} \label{Prop:Horizontal2}

    Let $\g: [a,b] \to \T(S)$ be a Thurston geodesic and let $R\subseteq S$ be a
    subsurface. If $\bd R$ is horizontal on $X_a$, then $\diam_R(\g([a,b]))
    = O(1)$.
  
  \end{proposition}

  \begin{proof}

    If the boundary $\bd R$ is horizontal at $X_a$, then  it is either
    weakly horizontal or $(n_0,L_0)$--horizontal. Then we choose a time
    $s\in [a,b]$ as follows:
    \begin{itemize}
   \item If $\alpha \in \bd R$ is weakly horizontal on $X_a$ with anchor curve
    $\gamma$, then let $s\in [a,b]$ be, if exists, the first time such that
    $\ell_s(\gamma) = 1$.  If no such $s$ exists, let $s=b$. 
    \item If $\alpha$ is $(n_0,L_0)$--horizontal on $X_a$, let $s=a$. 
    \end{itemize}
   
    Now note that along the interval $[a,s]$ the curve $\gamma$ remains
    $1$--short so $$\diam_R(\g([a,s])) = O(1).$$ So we only need to show
    that $\diam_R(\g([s,b])) = O(1)$; we can assume $b-s \ge s_0$,
    otherwise we are done. First note that when $s$ is chosen as the first
    bullet by Theorem \ref{Thm:Whor}, $\alpha$ is $(n_0,L_0)$--horizontal
    on $X_s$ and when $s$ is chosen as the second bullet $\alpha$ is
    obviously $(n_0,L_0)$--horizontal on $X_s$. So in any case  $\alpha$ is
    obviously $(n_0,L_0)$--horizontal on $X_s$. Now let $\gamma$ be an
    $\ep_B$--short curve on $X_s$ crossing $R$. By Proposition
    \ref{Prop:Growth}, for all $t$ that $t-s \ge s_0$, $\ell_t(\alpha) \ge
    e^{t-s} w_{\ep_h}$, and $\ell_t(\gamma) \le e^{t-s} \ep_B$. Thus, by
    \propref{RelTwist},
    \[ 
      d_R(X_s,X_t) \le d_R(\gamma,X_t) \prec \log
      \frac{\ell_b(\gamma)}{\ell_t(\alpha)} \le \log \frac{\ep_B}{w_{\ep_h}}. 
    \]
    This is true for all $t$ with $t-s \ge s_0$, so $\diam_R(\g([s,b])) = O(1)$ as
    desired. \qedhere
  
  \end{proof}

  \subsection{Corridors}

  In this section, we recall a tool for constructing horizontal curves on
  hyperbolic surfaces assuming certain conditions on intersection numbers
  between curves. This is the notion of corridors on hyperbolic surfaces as
  was introduced in \cite[\S 5.1]{LRT2}. 

  \begin{definition}[Corridor]\label{Def:Corridor}
    Let $X \in \T(\s)$ and let $(\tau,\omega)$ be an ordered pair of simple
    geodesic segments. An $(n,L)$--corridor generated by $(\tau,\omega)$ is
    the image $Q$ of a map $f:[0,a]\times [0,b]\to X$ such that
    \begin{itemize}
      \item $f$ is an embedding on $(0,a) \times (0,b)$. 
      \item Edges $[0,a] \times \{0\}$ and $[0,a]\times \{b\}$ are mapped
        to a segment of $\omega$ and the interior of $Q$ does not meet
        $\omega$.
      \item There is a partition $0=t_0<\ldots < t_{n+1}=a$ such that each
        $\{t_i\}\times [0,b]$ is mapped to a segment of $\tau$, for all
        $i=0,\ldots, n+1$.
      \item Segments $f([t_i,t_{i+1}])\times  \{0\})$ and $f( [t_i,t_{i+1}]
        \times\{b\})$ have lengths at least $L$.
    \end{itemize}
    
    We will call $f(\{0\} \times [0,b])$ and $f(\{a\} \times [0,b])$ the
    \emph{vertical} sides of $Q$, and $f([0,a] \times \{0\})$ and $f([0,a]
    \times \{b\})$ the \emph{horizontal} sides of $Q$. 

  \end{definition}

  The following lemma guarantees the existence of a corridor generated by
  two curves given that the curves intersect many times inside a
  subsurface.
  
  \begin{lemma}\label{Lem:Corridor1}
    Given $n \in \mathbb{N}$ and $L>0$, let
    \[
      C(n,L):=(6|\chi(\s)|+1)n \left\lceil \frac{L}{2w_{\ep_h}} \right
      \rceil+3|\chi(\s)|+1.
    \]
    Then given $X \in \T(\s)$ and $R\subseteq S$ subsurface, let $\gamma$ be an
    $\ep_h$--short curve intersecting $R$ and let $\tau$ be a component of
    $\gamma \cap R$. If $\omega$ is a simple geodesic segment in $R$ with
    $\I_R(\tau,\omega) \geq C(n,L)$, then there exists an $(n,L)$--corridor
    $Q$ in $R$ generated by $(\tau,\omega)$. Moreover, there exists a curve
    $\alpha \in \calC(R)$ such that $\I_R(\tau,\alpha) \le C(n,L)$.
    Moreover, if $\beta$ is any curve or arc that does not cross the
    interior of $Q$ from vertical to vertical sides, then
    $\I_R(\beta,\alpha) \le \I_R(\beta,\omega)+1$.

  \end{lemma}
  
  \begin{proof}
    
    With the exception of the last statement, the proof of the lemma is
    identical to that of \cite[Lemma 5.4]{LRT2}. The only adjustment is the
    replacement of the constant $\delta_B = 2 w_{\ep_B}$ by $2w_{\ep_h}$.
    Since $\tau$ and $\omega$ both lie in $R$, the corridor they generate
    will also lie in $R$. Finally, the existence of the curve $\alpha$ can
    be found in the proof of \cite[Proposition 5.5]{LRT2}. \qedhere
    
  \end{proof}
  
  The following lemma will be useful later to determine when a curve is
  horizontal. The proof is immediate by lifting the picture to the
  universal cover.

  \begin{lemma} \label{Lem:Corridor2}

    Suppose $\g \colon I \to \T(\s)$ is a Thurston geodesic. For $t \in I$,
    suppose there is an $(n_0,L_0)$--corridor $Q$ in $X_t$ whose vertical
    sides belong to an $\ep_h$--short curve $\gamma$. Then the following
    statements hold.
    \begin{itemize}
      \item If a leaf of $\lambda_\g$ crosses the interior of $Q$ from
        vertical to vertical sides, and the horizontal sides of $Q$ belong
        to a curve $\alpha$, then $\alpha$ is $(n_0,L_0)$--horizontal on
        $X_t$ with anchor curve $\gamma$. 
      \item If the horizontal sides of $Q$ belong to a leaf of $\lambda_\g$
        or a horizontal segment of a curve, then any curve that crosses the
        interior of $Q$ from vertical to vertical sides is
        $(n_0,L_0)$--horizontal on $X_t$ with anchor curve $\gamma$. 
    \end{itemize}

  \end{lemma}

\section{Active intervals}  

  \label{Sec:AI}

  Let $\g \colon I \to \T(\s)$ be a Thurston geodesic and let $R\subseteq
  S$ be a subsurface. In this section, we introduce the notion of {\em an
  active interval} $J_R \subseteq I$ of $R$ along $\g$ with the property
  that for all $t \in J_R$, the length of $\bd R$ is uniformly bounded
  (independently of $\g, t$ and $R$) and the restriction of $X_t=\g(t)$ to
  $R$ resembles a geodesic in the \Teich space of $R$. In particular, we
  will show that, similar to the main result of \cite{LRT2}, $\g$
  restricted to $J_R$ projects to a reparametrized quasi-geodesic in
  $\calC(R)$. Furthermore, outside of $J_R$, $\g$ does not make much
  progress in $\calC(R)$.  These statements will be made precise within
  this section. Theorem \ref{thm:main} in the introduction which is the
  main result of the paper is a combination of Theorem \ref{Thm:AI} and
  Corollary \ref{Cor:Reparam} in this section. 

  Let $\ep_h$, $\ep_w$, $n_0$ and $L_0$ be the constants that we have fixed
  in the previous section (Lemma \ref{Lem:Constants}). Recall that $\bd R$
  is \emph{horizontal} on $X_t$ if a component of $\bd R$ is either weakly
  or $(n_0,L_0)$--horizontal on $X_t$ (see Definition \ref{Def:Shor}). We
  first show the following trichotomy about $\bd R$ along Thurston
  geodesics. 

  \begin{proposition}\label{Prop:Trichotomy}

    There exists a constant $\rho > 0 $ such that, for any Thurston
    geodesic $\g:I\to\calT(S)$, any subsurface $R \subseteq\s$ intersecting
    $\lambda_\g$ and any $t\in I$, one of the following three statements
    holds:
    \begin{enumerate}[label=\textup{(T\arabic*)}]
      \item[(T1)] $\ell_t(\bd R) = \max_{\alpha \in \bd R}
        \ell_t(\alpha) \leq \rho$. \label{T1} 
      \item[(T2)] $\bd R$ is horizontal on $X_t$. \label{T2}
      \item[(T3)] $d_R(X_t,\lambda_\g) \le 1$. \label{T3}
    \end{enumerate}
           
  \end{proposition}

  \begin{proof} 
        
    To start the proof recall that by Lemma \ref{Lem:Long}, there exists a
    constant $L_{\ep_w}>0$ such that any geodesic of length $L_{\ep_w}$ on
    an $\ep_w$--thick subsurface $Y\subseteq \s$ must intersect the short
    marking $\mu_Y$ on $Y$. Define the number
    \[ L = (4n_0\, \xi(S) + 2)\max \set{L_{\ep_w},L_0}.\] 
    Also define the numbers
    \[ \delta = \frac{2\pi|\chi(S)|}{L}, \qquad K =
    \frac{|\chi(S)|}{\sinh^2(\delta/2)}, \qquad \rho = 2KL.\] 

    Now let $\lambda$ be a leaf of $\lambda_\g$ intersecting $R$
    essentially. Suppose first that a component $\alpha \subset \bd R$
    intersects an $\ep_w$--short curve $\gamma$. If $\lambda$ does not
    intersect $\gamma$, then $d_R(X_t,\lambda_\g) \leq d_R(\gamma,\lambda)
    \le 1$, so (T3) holds. Now suppose that $\lambda$ intersects $\gamma$.
    If $\alpha$ is not weakly horizontal, then inside the standard collar
    neighborhood $U(\gamma)$ of $\gamma$, every arc of $\alpha \cap
    U(\gamma)$ must intersect every arc of $\lambda \cap U(\gamma)$. In
    this case, we can find an arc of $\omega \subseteq\lambda \cap R$, such
    that $\omega$ has at least one endpoint on an arc of $\alpha \cap
    U(\gamma)$, and then homotope $\omega$ (relative to $\bd R$) to be
    disjoint from $\gamma$. This shows $d_R(X_t,\lambda_\g)\le 1$ so again
    (T3) holds. 
    
    So we may assume that $\bd R$ does not intersect any $\ep_w$--short
    curve. Let $\alpha$ be a component of $\bd R$ with $\ell_t(\alpha) >
    \rho$. Then we can find an $\ep_w$--thick subsurface $Y$ containing
    $\alpha$ such that either $\bd Y = \emptyset$ or $\ell_t(\bd Y) \le
    \ep_w$. A short marking $\mu_Y$ on $Y$ is a subset of the short marking
    on $X$; in particular, every curve of $\mu_Y$ is $\ep_h$--short. Since
    $\ell_t(\alpha) > \rho$, by Lemma \ref{Lem:Quad}, there exists an
    $L$--wide quadrilateral $Q$ in $R$, whose top edge is a segment of
    $\alpha$ and bottom edge belongs to $\partial R$. By subdividing $Q$,
    we can find a collection of quadrilaterals $Q_1, Q_2, \ldots, Q_n
    \subseteq R$ with the following properties: 
    \begin{itemize}
      \item $n = 4n_0\, \xi(S)+2$
      \item $Q = \bigcup_{i=1}^n Q_i$, and the indices are arranged so that
        the right side of $Q_i$ is the left side of $Q_{i+1}$ for all $i
        =1, \ldots, n-1$. 
      \item The top edge of $Q_i$ is a segment of $\alpha$ with length at
        least $\max \set{L_{\ep_w},L_0}$.  
    \end{itemize}
    By the third property above, for each $i$, there exists a curve
    $\gamma_i$ in $\mu_Y$ that intersects the top edge of $Q_i$. Since all
    curves in $\mu_Y$ are $\ep_h$ short and $L_0 > \ep_h$, if the curve
    $\gamma_i$, for $i =2,\ldots, n-1$, intersects the top edge of $Q$,
    then it must cross the quadrilateral $Q_i$, and exit the bottom edge of
    $Q_j$ for some $j \in \{i-1,i,i+1\}$. That is, $\gamma_i \cap Q$ has a
    component $\delta_i$ which goes from the top edge of $Q_i$ to the
    bottom edge of $Q_j$, $j \in \{i-1,i,i+1\}$. Then, since $n=4n_0
    \,\xi(\s)+2$ and there are at most $2|\xi(S)|$ curves in $\mu_Y$, at
    least $2n_0$ of the $\delta_i$'s belong to the same curve $\gamma$ in
    $\mu_Y$. By skipping every other component if necessary, we can find
    $n_0$ arcs in $Q$ belonging to $\gamma$ with pairwise distance at least
    $L_0$. If $\lambda$ does not cross $Q$ from side to side, then there is
    an arc of $\lambda \cap R$ that misses an arc of some $\gamma_i\in
    \mu_Y$ with end point on $\bd R$, thus $d_R(X_t, \lambda_\g) \le 1$,
    which means that (T3) holds. On the other hand, if $\lambda$ crosses
    $Q$ from side to side, then $\alpha$ is $(n_0,L_0)$--horizontal with
    anchor curve $\gamma$ and hence (T2) holds. This finishes the proof of
    the trichotomy. \qedhere 

  \end{proof}
 
  Let $\rho>0$ be the constant from \propref{Trichotomy}. 

  In preparation for the next section, we now prove two technical
  propositions about when curves are horizontal in a subsurface $R$.
  
  \begin{proposition}\label{Prop:Lipschitz} 
    
    There exists an integer $N_0\geq n_0$ such that the following statement
    holds. Let $\g \colon I \to \T(\s)$ be a Thurston geodesic, $R \subseteq
    S$ a subsurface, and $X=\g(t)$ a point with $\ell_X(\bd R) \le \rho$.
    If a curve $\alpha \in \calC(R)$ is $(n,L)$--horizontal on $X$, where
    $n \ge N_0$, $L \ge L_0$, then any curve $\beta \in \calC(R)$ that
    intersects $\alpha$ at most twice is either $(n_0,L_0)$--horizontal on
    $X$ or $d_R(X,\beta) = O(1)$. 
  
  \end{proposition}
 
  \begin{proof}

    Let $C=C(n_0,L_0)$ be the constant from Lemma
    \ref{Lem:Corridor1} and let $N_0=\frac{\ep_h}{2w_\rho} C$. Assume that
    $\alpha$ is $(n,L)$--horizontal on $X$, where $n \ge N_0$ and $L \ge
    L_0$.

    Let $\gamma$ be an anchor curve for $\alpha$. Since $\alpha$ lies in
    $R$, $\gamma \cap R$ is non-empty. By definition, $\gamma$ is
    $\ep_h$--short, so $\I_R(X,\gamma) = O(1)$. Let $\talpha$ be an
    $(n,L)$--horizontal lift of $\alpha$. Let $q_1,\ldots,q_n$ be the
    points along $\talpha$ given by the definition of $(n,L)$--horizontal.
    Let $\tomega$ be the innermost segment of $[q_1,q_n)$ that intersects
    exactly $N_0$--lifts of $\tgamma$ and let $\omega$ be the projection of
    $\tomega$ to $X$. If $\omega = \alpha$, then $\I_R(\gamma,\alpha) \le
    N_0$. In this case, by the triangle inequality and (\ref{Eqn:di}), we obtain 
    \begin{eqnarray*} 
      d_R(X,\beta) &\ladd&
      d_R(X,\gamma)+d_R(\gamma,\alpha)+d_R(\alpha,\beta)\\ &\lmul& \Log
      \I_R(X,\gamma) + \Log \I_R(\gamma,\alpha) + d_R(\alpha,\beta) \lmul
      \Log N_0,
    \end{eqnarray*}
    which is one of the conclusions of the proposition.

    Now suppose $\omega$ is a proper arc of $\alpha$. By construction,
    $\I_R(\gamma,\omega) = N_0$. Moreover, $\ell_X(\bd R) \le \rho$, so any
    arc in $\gamma \cap R$ is at least $2 w_\rho$ long. Thus, $\gamma \cap
    R$ has at most $\ep_h/2w_\rho$ arcs, and hence there is at least one
    arc $\tau$ with $\I_R(\tau,\omega) \geq C$. By Lemma
    \ref{Lem:Corridor1}, $\tau$ and $\omega$ generates a
    $(n_0,L_0)$--corridor $Q$ in $R$. If $\beta$ crosses the interior of
    $Q$ from vertical to vertical side, then $\beta$ is
    $(n_0,L_0)$--horizontal with anchor curve $\gamma$ by
    \lemref{Corridor2}. Otherwise, let $\alpha' \in \calC(R)$ be the curve
    guaranteed by \lemref{Corridor1}. The curve $\alpha'$ intersects
    $\beta$ at most twice and $\I_R(\tau, \alpha') \le \I_R(\tau,\omega)
    \le N_0$. Then by the triangle inequality and (\ref{Eqn:di}) we obtain
     \[ d_R(X_s,\beta) \ladd d_R(X_s,\alpha')+d_R(\alpha',\beta)
      \lmul \Log \I_R(\tau,\alpha') \le \log N_0.\qedhere \]
  \end{proof}
  
  The following proposition gives us a pre-horizontal or strongly
  horizontal curve of bounded length in a subsurface $R$ along a Thurston
  geodesic, provided that $\partial R$ is sufficiently short and not
  horizontal. It is a generalization of \cite[Proposition 4.5]{LRT2} to
  subsurfaces. 

  \begin{proposition} \label{Prop:Retraction}
     
    There exists a constant $N_1\in \mathbb{N}$ such that the following
    statement holds. Let $\g \colon I \to \T(\s)$ be a Thurston geodesic
    and let $X=\g(t)$ be a point such that $\ell_X(\bd R) \le \rho$ and
    $\partial R$ is not horizontal on $X$. If $\I_R(X,\lambda_\g) \ge N_1$,
    then there exists a curve $\alpha \in \calC(R)$ of uniformly bounded length which
    is either pre-horizontal or $(n_0,L_0)$--horizontal on $X$.

  \end{proposition}

  \begin{proof}

    Let $K\geq 1$ be the additive error from Proposition \ref{prop :
    enhance to horizontal} that corresponds to the maximum of the fellow
    traveling distances in the statements of  Lemma 5.8 and Lemma 5.9 of
    \cite{LRT2}. Then, let $N= n_0\, \left\lceil \frac{L_0}{2w_{\ep_B}}
    \right\rceil +K$ and $N_1 = 5N$. We also fix a small constant $\ep_\rho
    \le \ep_w$ so that  $w_{\ep_\rho} > \rho$. Note that, by our choice
    $\ep_\rho$, any $\ep_\rho$--short curve of $X$ is disjoint from
    $\partial R$. 
    
    If there is an $\ep_\rho$--short curve in $R$ that intersects
    $\lambda_\g$, then by definition this curve is pre-horizontal curve and
    we are done. Thus, we may proceed with the case that no such curve
    exists. 

    Let $P$ be the set of $\ep_B$--short curves on $X$, and suppose there
    exists a leaf $\lambda$ of $\pi_R(\lambda_\g)$ that intersects some
    element of $\pi_R(P)$ at least $N_1$ times. Take a segment $\blambda
    \subseteq \lambda$ and an arc or curve $\tau \in \pi_R(P)$ such that
    $\blambda$ has endpoints on $\tau$, $\I_R(\tau,\blambda) = N_1$, and
    $\I_R(\tau,\blambda)$ is maximal among all other elements in
    $\pi_R(P)$. Note that $\blambda$ does not intersect any
    $\ep_\rho$--short curves, so its length is uniformly bounded. If
    $\blambda$ closes up to a curve, then we have found our
    $(n_0,L_0)$--horizontal curve. Otherwise, we will find such a curve by
    applying a surgery between $\tau$ and $\blambda$. This procedure is
    involved and will depend on the intersection pattern between $\tau$ and
    $\blambda$. The conclusion will follow if either of the following three
    cases occur.
    
    {\bf Case (1)}: There is a subsegment $\omega$ of $\blambda$ that hits
    $\tau$ on opposite sides with $\I_R(\tau,\omega) \ge N$, and the
    endpoints of $\omega$ can be joined by a segment of $\tau$ that is
    disjoint from the interior of $\omega$. 

    In this case, let $\alpha$ be the curve obtained from the concatenation
    of $\omega$ and a segment of $\tau$ disjoint from the interior of
    $\omega$. Then $\alpha$ is $(n_0,L_0)$--horizontal with anchor curve
    $\gamma$. To see this, note that by Lemma 5.8 of \cite{LRT2} $\talpha$
    a lift of $\alpha$ stays in a bounded neighborhood of $\tomega$ and
    $\teta$ lifts of $\omega$ and $\eta$. Then since $\gamma$ and $\omega$
    intersect at least $N_1$ times ($\tau\subseteq \gamma$) and $N_1>N$ at
    least $N$ lifts of $\gamma$ intersect $\tomega$. Moreover, by the
    choice of $N$, skipping some of the lifts of $\gamma$, we may choose at
    least $n_0+K$ lifts of $\gamma$ so that the distance between the
    intersection points between any two consecutive lifts and $\tomega$ is
    at least $L_0$. But then by Proposition \ref{prop : enhance to
    horizontal} there are at least $n_0$ lifts of $\gamma$ that intersect
    both $\tomega$ and $\talpha$.  This shows that $\alpha$ is
    $(n_0,L_0)$--horizontal (see the definition of horizontal).
       
    {\bf Case (2)}: There is a subsegment $\omega$ of $\blambda$ and a
    closed curve $\beta$ disjoint from $\omega$, such that $\I_R(\tau,
    \omega) \ge N$ and $\ell_X(\beta) \ge 2w_{\ep_B}$, and the endpoints of
    $\omega$ are close to the same point on $\beta$. Furthermore, each
    endpoint of $\omega$ can be joined to a nearby point on $\beta$ by a
    segment of $\tau$ that is disjoint from $\beta$ and the interior of
    $\omega$.
      
    In this case let $\alpha$ be the curve obtained by closing up $\omega$
    with $\beta$ and one or two subsegments of $\tau$. By Lemma 5.9 of
    \cite{LRT2}  $\tomega$ a lift of $\omega$ stays in a bounded
    neighborhood of $\talpha^*$ a lift of $\alpha^*$ (the geodesic
    representative of $\alpha$) and two lifts of $\beta$. Moreover, since
    $\omega$ and $\gamma$ intersect at least $5N$ times, the lifts of
    $\gamma$ intersect either $\talpha^*$ or a lift of $\beta$ at least $N$
    times. To see this, note that the lifts of $\gamma$ do not intersect
    the lifts of subsegments of $\tau$ ($\tau\subseteq \gamma$), so the
    lift of $\omega$ and two lifts of $\beta$ are intersected by the $5N$
    lifts of $\gamma$.

    In the former situation applying Proposition \ref{prop : enhance to
    horizontal} and by the choice of $N$ we can see that $\alpha$ is
    $(n_0,L_0)$--horizontal with anchor curve $\gamma$. In the later
    situation $\beta$ is $(n_0,L_0)$--horizontal by Proposition \ref{prop :
    enhance to horizontal} with anchor curve $\gamma$. In fact, in this
    case $\omega$ is twisting around $\beta$.
    

    {\bf Case (3)}: There is a subsegment $\omega$ of $\blambda$ and two
    closed curves $\beta$ and $\beta'$ disjoint from $\omega$, such that
    $\I_R(\tau, \omega) \ge N$, $\ell_X(\beta) \ge 2w_{\ep_B}$ and
    $\ell_X(\beta') \ge 2w_{\ep_B}$, and the two endpoints of $\omega$ are
    close to $\beta$ and $\beta'$. Furthermore, $\tau$ has a segment
    joining one endpoint of $\omega$ to $\beta$ and another segment joining
    the other endpoint of $\omega$ to $\beta'$, such that both segments are
    disjoint from $\beta$, $\beta'$ and the interior of $\omega$.

    In this case, let $\alpha$ be the curve obtained by gluing two copies
    of $\omega$, $\beta$, $\beta'$ and a few subsegments of $\tau$. By
    Lemma 5.9 of \cite{LRT2} $\tomega$ a lift of $\omega$ stays in a
    bounded neighborhood of $\talpha^*$ a lift of $\alpha^*$, $\tbeta$ a
    lift of $\beta$ and $\tbeta'$ a lift of $\beta'$. Then similar to case
    (2) at least $N$ lifts of $\gamma$ intersects $\talpha$, $\tbeta$ or
    $\tbeta'$. Thus again appealing to Proposition \ref{prop : enhance to
    horizontal} and by the choice of $N$ either $\alpha, \beta$ or $\beta'$
    is $(n_0,L_0)$--horizontal.

    We now show that at least one of the above three cases happens.
    
    Let $p$ and $q$ be the endpoints of $\blambda$. Consider the
    intersection points between $\blambda$ and $\tau$ ordered along $\tau$.
    If each $p$ and $q$ is adjacent to two intersection points, then this
    proof is contained in \cite[Proposition 5.10]{LRT2}. In the following,
    we will consider the case that at least one of $p$ or $q$ is not
    adjacent to two intersection points. Note that when this happens,
    $\tau$ must be an arc in $R$ and any component of $\bd R$ crossing
    $\tau$ has length at least $2 w_{\ep_B}$.
    
    First suppose that neither $p$ nor $q$ is adjacent to two intersection
    points (see Figure \ref{Fig : Surgery1}). In this case, the
    intersection points along $\tau$ line up from $p$ to $q$. Let $\beta$
    and $\beta'$ be the two (possibly non-distinct) components of $\bd R$
    containing the endpoints of $\tau$. By relabeling if necessary, assume
    $\beta$ is closest to $p$ and $\beta'$ is closest to $q$. Let $\alpha
    \in \calC(R)$ be the curve obtained by gluing two copies of $\blambda$,
    $\beta$, $\beta'$, and the subsegments of $\tau$ from $p$ to $\beta$
    and from $q$ to $\beta'$. Similar to Case (3), either $\alpha$,
    $\beta$, or $\beta'$ is $(n_0,L_0)$--horizontal. However, since by
    assumption $\bd R$ is not horizontal, it must be that $\alpha$ is the
    desired $(n_0,L_0)$--horizontal curve. 

    \begin{figure}[htp!]
      \begin{center}
      \includegraphics{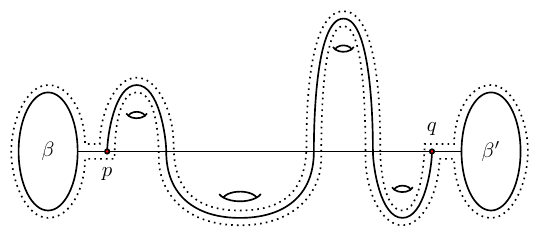}
      \caption{Neither $p$ nor $q$ is adjacent to two intersection points
        of $\bar\lambda$ and $\tau$.}
      \label{Fig : Surgery1} 
      \end{center}
     
    \end{figure} 

    Now assume exactly one of $p$ or $q$ is adjacent to two intersection
    points. Without the loss of generality, assume $p$ is adjacent to $p_1$
    and $p_2$. In this case, the intersection points along $\tau$ ends with
    $q$ at one end. After relabeling if necessary, assume $\blambda$ passes
    from $p$ to $p_1$ and then to $p_2$. (We include the possibility that
    $p_2 = q$.) Let $\omega_1$ be the subsegment of $\blambda$ from $p$ to
    $p_1$ and $\omega_2$ be the subsegment from $p_1$ to $p_2$. Among the 3
    segments $\omega_1$, $\omega_2$, and $\omega_1 \cup \omega_2$, there is
    at least one that hits $\tau$ on opposite sides. Pick the one that has
    maximal intersection number with $\tau$ and call it $\eta$. Let $\beta$
    be the closed curve obtained from closing up $\eta$ with a segment of
    $\tau$. Since $\beta$ stays close to $\eta$ by Lemma 5.8 of
    \cite{LRT2}] and was constructed from a segment with endpoints on
    $\tau$, $\ell_X(\beta) \ge 2w_{\ep_B}$. If $\I_R(\tau,\eta) \ge N$,
    then we are in case (1) with $\alpha = \beta$. If $\I_R(\tau,\eta) <
    N$, then there is a segment $\omega$ of $\blambda$ such that
    $\I_R(\tau,\omega) \ge N$, $\omega$ shares one endpoint with $\eta$,
    and the other endpoint of $\omega$ is either $p$ or $q$. If the other
    endpoint of $\omega$ is $p$, then let $\omega = \omega_1$. In this
    case, there are segments of $\tau$ connecting the endpoints of $\omega$
    to $\beta$ which are disjoint from the interiors of $\beta$ and
    $\omega$, so we can do a surgery as in case (2) to obtain the desired
    curve $\alpha$. Finally, if $q$ is an endpoint of $\omega$, then let
    $\beta'$ be the component of $\bd R$ connected to $q$ by a segment in
    $\tau$ which is disjoint from the interior of $\omega$ and $\beta'$.
    There is also a segment of $\tau$ connecting the other endpoint of
    $\omega$ with $\beta$ which is disjoint from their interior. Now we are
    in case (3), where the curve $\alpha$ is obtained by gluing two copies
    of $\blambda$, $\beta$, $\beta'$, and these subsegments of $\tau$. In
    this case, since $\bd R$ is not horizontal, either $\alpha$ or $\beta$
    is $(n_0,L_0)$--horizontal. In either case, we have found our
    $(n_0,L_0)$--horizontal curve in $R$.

    Finally, the $(n_0,L_0)$--horizontal curve $\alpha$ that we have
    constructed is a concatenation of a segment of $\blambda$ (possibly
    traversed twice), at most 4 segments of $\tau$, and at most 2
    components of $\bd R$. All of these arcs and curves involved have
    uniformly bounded length, whence so does $\alpha$. 
    This finishes the proof of the proposition. \qedhere
 
  \end{proof}

  \subsection{Active interval}

  \label{subsec:actint}

  In this subsection we introduce the notion of an active interval which is
  central to this paper. Then we will prove some results about the behavior
  of Thurston geodesics over active intervals.
  
  \begin{lemma}[Constant B]\label{Lem:ConstantB}
    Let $C=C(n_0,L_0)$ be the constant from \lemref{Corridor1}, $s_0$ the
    constant from Theorem \ref{Thm:Horizontal}, and $s_1=\log
    \frac{\rho}{2w_{\ep_h}}$. Then there is a constant $B> 3$ such that. 
  \begin{itemize}
    \item For all $R \subseteq \s$ and $\alpha,\beta\in\calC(R)$, if
      $\I_R(\alpha,\beta) \le C$, then $d_R(\alpha,\beta) \le B$.
    \item For all $X,Y \in \T(S)$ and all $R\subseteq S$, if $d(X,Y) \le
      \max \{ s_0, s_1 \}$, then $d_R(X,Y) \le B$.
    \item If $\alpha$ is $\rho$--short on $X$, then $d_S(X,\alpha) \le B$.
  \end{itemize}
  \end{lemma}
  
  \begin{proof}
    First note that by (\ref{Eqn:di}) $d_R(\alpha,\beta)\lmul\Log
    \I_R(\alpha,\beta)$, thus the first bullet holds for a $B$ that is
    larger than a multiple of  $\Log C$. 

    For the second bullet point, recall that the shadow map $\T(S) \to
    \calC(R)$ is coarsely-Lipschitz with Lipschitz constant that only
    depends on $S$. Thus such $B$ exists for a fixed $s_1$.

    Finally, since $d(X,\alpha)\lmul \Log\I(\beta,\alpha)$ where $\beta$ is
    an $\ep_B$--short curve on $X$ and $\I(\beta,\alpha)\leq
    \frac{\rho}{2w_{\ep_B}}$, so the last bullet point holds when $B$ is
    larger than a multiple of $ \Log \frac{\rho}{2w_{\ep_B}}$. \qedhere

  \end{proof}
  
  For the rest of the paper, we will fix a constant $B>6$ as in the above
  lemma, and recall the constant $\rho$ of Proposition
  \ref{Prop:Trichotomy}. We establish the following statement.  

  \begin{lemma} \label{Lem:AI}
     
    Given a Thurston geodesic $\g:[a,b]\to\calT(S)$ and a subsurface
    $R\subseteq S$ that intersects $\lambda_\g$. For $s \in [a,b]$, if
    $d_R(X_s,\lambda_\g) \ge 2B$, then for all $t \ge s$,
    $d_R(X_t,\lambda_\g) > 1$.  

  \end{lemma}

  \begin{proof}

    Let $P_t$ be the set of $\ep_B$--short curves on $X_t$. Assume that
    $d_R(X_t,\lambda_\g) \le 1$. Then by the triangle inequality
    (\ref{eq:triineq}) and that $B>3$ we deduce that
    \begin{align} \label{Eqn:AI} 
      d_R(X_s,X_t) \geq d_R(X_t,\lambda_\g) - d_R(\lambda_\g,X_s)-2
      \ge 2B-3 > B.
    \end{align} 
    Thus, there are $\tau \in \pi_R(P_s)$, $\omega \in \pi_R(P_t)$, and
    $\lambda \in \pi_R(\lambda_\g)$ such that $\I_R(\tau,\omega) \ge C$,
    and $\lambda$ is disjoint from $\omega$. By Lemma \ref{Lem:Corridor1},
    $\tau$ and $\omega$ generate an $(n_0,L_0)$--corridor $Q$, whose
    horizontal sides are disjoint from $\lambda$. Thus, either $\lambda$ is
    disjoint from $Q$ or it crosses the interior of $Q$ from vertical to
    vertical sides. In the latter case, the curve $\gamma \in P_t$
    containing $\omega$ is $(n_0,L_0)$--horizontal on $X_s$ by Lemma
    \ref{Lem:Corridor2}. Since $\gamma$ is $\ep_B$--short on $X_t$, by
    Proposition \ref{Prop:Growth}, $t-s \le \log \max \left\{s_0, s_1
    \right\}$, but this implies $d_R(X_s,X_t) \le B$, contradicting
    Equation \ref{Eqn:AI}. On the other hand, if $\lambda$ is disjoint from
    $Q$, then let $\alpha$ be the curve guaranteed by Lemma
    \ref{Lem:Corridor1}, which has the property that $\I_R(\tau,\alpha) \le
    C$ and $\I_R(\alpha,\lambda) \le 1$. By the triangle inequality, this
    yields $$d_R(X_s,\lambda_\g) \leq d_R(\tau,\alpha)+
    d_R(\alpha,\lambda) +2\le B+3<2B,$$ contradicting the assumption.
    This shows $d_R(X_t,\lambda_\g) > 1$. \qedhere 

  \end{proof}

  We are now ready to give the definition of the {\em active
  interval} of a subsurface $R\subsetneq S$ along a Thurston geodesic $\g$. 
  
  \begin{definition}[Active Interval]\label{Def:AI} 
   
    Let $\g \colon [a,b] \to \T(S)$ be a Thurston geodesic and let
    $R\subseteq S$ be a subsurface that intersects $\lambda_\g$. We define
    the \emph{active interval} $J_R=[c,d] \subseteq [a,b]$ of $R$ along
    $\g$ as follows. The right endpoint of the interval is 
    \[ d := \inf \Big\{t \in [a,b] \colon \bd R \text{ is
    horizontal on } X_t \Big\}, \] 
     and if $\bd R$ is never horizontal along $[a,b]$, then we set $d=b$.
    The left endpoint of the interval is 
    \[ c := \inf \Big\{ t \in [a,d] \colon d_R(X_t,\lambda_\g) \ge 2B\Big\},
     \] 
    and if $d_R(X_t,\lambda_\g)< 2B$ for all $t\in [a,d]$, then we set $c=d$. 

  \end{definition}

  \begin{theorem} \label{Thm:AI}
   
    Given a Thurston geodesic $\g:[a,b]\to\calT(S)$ and a subsurface
    $R\subseteq S$ that intersects $\lambda_\g$, the active interval $[c,d]
    \subseteq [a,b]$ of $R$ satisfies the following properties.

    \begin{enumerate}[label=\textup{(\roman*)}]
      \item $\diam_R(\g([a,c]))$ and $\diam_R(\g([d,b]))$ are uniformly
        bounded. 
        In particular, if $d_R(X_a,X_b)$ is sufficiently large, then
        $[c,d]$ is a non-trivial interval. \label{ai1} 
      \item If $[c,d]$ is a non-trivial interval, then $\ell_t(\bd R)
        \leq \rho$ for all $t \in [c,d]$. \label{ai2}
    \end{enumerate}
  \end{theorem}

  \begin{proof}

    For all $s,t \in [a,c]$, by the choice of $c$ and the triangle
    inequality (\ref{eq:triineq}) we have
    \[d_R(X_s,X_t) \ladd d_R(X_s,\lambda_\g) +
    d_R(\lambda_\g,X_t) \le 6B,\] so $\diam_R(\g([a,c]) = O(1)$. Moreover,
    if $d \leq b$, then $\bd R$ is horizontal on $X_d$ by definition. Then
    by \propref{Horizontal2} we have $\diam_R(\g([d,b]) = O(1)$, thus
    \ref{ai1} is proved.
    
    For property \ref{ai2}, if $c < d$, then by our choice of $c$,
    $d_R(X_c,\lambda_\g) \ge 2B$. By Lemma \ref{Lem:AI} $d_R(X_t,\lambda_\g)
    > 1$, for all $t \in [c,d]$. Thus, by Proposition
    \ref{Prop:Trichotomy}, either $\partial R$ is horizontal on $X_t$ or
    $\ell_t(\partial R) \le \rho$. But by our choice of $d$, we must have
    $\ell_t(\partial R) \le \rho$ for all $t \in [c,d)$. The latter bound
    extends to $t=d$ by continuity. \qedhere  
    
  \end{proof}

  \subsection{No backtracking along active intervals} 

  We now define the {\em balanced time of a curve in a subsurface
  $R\subseteq S$} which is a time in the active interval of $R$ which helps
  us to define a contraction map from $\calC(R)$ to the shadow of a
  Thurston geodesic in $\calC(R)$.
  
  \begin{definition}[Balanced time of a curve]\label{def:balancedtime} 
    
    Let $\g \colon I \to \T(\s)$ be a Thurston geodesic and let $R$ be a
    subsurface intersecting $\lambda_\g$. Also, let $J_R=[c,d]$ be the
    active interval of $R$ along $\g$. Then, for a curve $\alpha \in
    \calC(R)$, the \emph{balanced time} of $\alpha$ in $[c,d]$ is the time
    \[ t_\alpha := \inf \Big\{ t \in [c,d] \colon \alpha \text { is }
    (n_0,L_0)\text{--horizontal on } X_t\Big\}. \] 
     
    But, if $\alpha$ is never $(n_0,L_0)$--horizontal along $[c,d]$, then
    set $t_\alpha = d$. 
  
  \end{definition}
  
  Now we define a coarse map $$\Pi \colon \calC(R)\to
  \Upsilon_R(\g([c,d]))$$ as follows. For any curve $\alpha \in \calC(R)$,
  let  $\Pi(\alpha) = \Upsilon_R(\g(t_\alpha))$, where $t_\alpha$ is the
  balanced time of $\alpha$ along \g. When $\tau$ is an arc in $R$, then
  set $\Pi(\tau) = \Pi(\alpha)$, where $\alpha$ is any disjoint curve from
  $\tau$ in $\calC(R)$. 
  
  Our main theorem of this section is the following.

  \begin{theorem}\label{Thm:LipRetract} 
    
    Let $\g \colon [a,b] \to \T(\s)$ be a Thurston geodesic and let $[c,d]
    \subseteq [a,b]$ be the active interval of a subsurface $R$ along $\g$.
    Then $\Pi \colon \calC(R)\to \Upsilon_R(\g([c,d]))$ is a $K$--Lipschitz
    retraction, where $K$ depends only on \s. More precisely, 
    \begin{itemize}
      \item $\Pi$ is Lipschitz: Given two curves $\alpha, \beta \in
        \calC(R)$ that intersect at most twice, \[
          d_R(X_{t_\alpha},X_{t_\beta}) \le K. \] In particular, $\Pi$ is
        coarsely well-defined, i.e.\ when $\tau$ is an arc, then
        $\Pi(\tau)$ is independent of the choice of a  curve in $\calC(R)$
        disjoint from $\tau$.
      \item $\Pi$ is a retraction: For any $t \in [c,d]$ and any curve $\alpha
        \in \calC(R)$ disjoint from any element of $\pi_R(X_t)$, we have
        $d_R(X_t,X_{t_\alpha}) \le K$. 
    \end{itemize}

  \end{theorem}

  \begin{proof}
    
    All statements hold trivially if $c=d$, so we will assume $c < d$. Then
    by part \ref{ai2} of \thmref{AI}, $\ell_t(\bd R) \le \rho$ for all $t
    \in [c,d]$.
    
    We first show that the map $\Pi$ is Lipschitz. Let $\alpha, \beta \in
    \calC(R)$ be two curves with $d_R(\alpha,\beta)=1$. Without a loss of
    generality, we may assume that $t_\alpha < t_\beta$. Let $N_0$ be the
    constant of \propref{Lipschitz} and let $s \in [t_\alpha,d]$ be the
    first time that $\alpha$ is $(N_0,L_0)$--horizontal; if no such time
    exists then set $s=d$. Then since $\alpha$ is $(n_0,L_0)$--horizontal
    at $X_t$ from the second part of \thmref{Horizontal} we can deduce that
    $\diam_R(\g([t_\alpha,s]))\ladd \log(\frac{N_0}{n_0})$. So if $t_\beta
    \in [t_\alpha,s]$, then we are done. If $t_\beta>s$, then it is enough
    to show that $d_R(X_s,X_{t_\beta}) = O(1)$. By definition of balanced
    time, $\beta$ is not $(n_0,L_0)$--horizontal along $[s,t_\beta)$. Now
    for any $t \in [s,t_\beta)$, if $\alpha$ remains
    $(N_0,L_0)$--horizontal on $X_t$ , then since the length of $\bd R$ is
    at most $\rho$ on $X_s$ and $X_t$ and $\alpha$ is
    $(N_0,L_0)$--horizontal on $X_s$ and $X_t$, we may apply
    \propref{Lipschitz} to get $$ d_R(X_s,X_t) \ladd d_R(X_s,\beta) +
    d_R(\beta,X_t) = O(1). $$ But, if $\alpha$ is not
    $(N_0,L_0)$--horizontal on $X_t$, then again by the second part of
    \thmref{Horizontal},  $t-s \ladd 0$. In either case, we have our
    desired bound for $d_R(X_t,X_s)$. 
   
    Now we show $\Pi$ is a retraction. Let $N_1$ be the constant of
    \propref{Retraction}. For any $t \in [c,d]$, let $P_t$ be the set of
    $\ep_B$--short curves on $X_t$. Note that there is a curve $\alpha$ of
    bounded length and disjoint from an element of $\pi_R(P_t)$; if
    $\pi_R(P_t)$ contains a curve, then let $\alpha$ be the curve.
    Otherwise, take an arc $\tau \in \pi_R(P_t)$ and do a surgery with $\bd
    R$ to get a curve $\alpha$ disjoint from $\tau$ of length at most
    $2(\rho+\ep_B)$. Since $\Pi$ is Lipschitz, to show that $\Pi$ is a
    retraction it suffices to show that $d_R(X_t,X_{t_\alpha}) = O(1)$.
  
    If $t_\alpha < t$, then in particular $t_\alpha <d$. By definition of
    the balanced time, $\alpha$ is $(n_0,L_0)$--horizontal on
    $X_{t_\alpha}$, so $\alpha$ crosses an $\ep_h$--short curve on
    $X_{t_\alpha}$ and hence $\ell_{t_\alpha}(\alpha) \ge 2w_{\ep_h}$.
    Then, since $\alpha$ has length at most $2(\rho+\ep_B)$ on $X_t$, by
    \propref{Growth}, $t-t_{\alpha}\leq \log \frac{\rho+\ep_B}{w_{\ep_h}}$.
    This shows that $d_R(X_t,X_{t_\alpha}) = O(1)$ when $t_\alpha < t$. 
    
    Now assume that $t_\alpha > t$. Let $s \in [t,d]$ be the first time
    that $\I_R(X_s,\lambda_\g) \ge N_1$; set $s = d$ if no such time
    exists. It is immediate from (\ref{Eqn:di}) that $\diam_R(\g([t,s]))
    \lmul \Log N_1$, so we would be done if $t_\alpha \in [t,s]$. Thus,
    assume that $t_\alpha \in (s,d]$. By \propref{Retraction}, there exists
    a bounded length curve $\beta \in \calC(R)$ which is either
    pre-horizontal or $(n_0,L_0)$--horizontal on $X_s$. In either case,
    $d_R(X_{t_\beta},X_s) = O(1)$. Indeed, if $\beta$ is pre-horizontal,
    then let $s_\beta > s$ be the first moment when $\ell_{s_\beta}(\beta)
    = 1$; set $s_\beta = d$ if no such time exists. Of course,
    $\diam_R(\g[s,s_\beta])=O(1)$. But $t_\beta  \in [s,s_\beta]$ by
    Proposition \ref{Prop:Growth} and Theorem \ref{Thm:Whor}, so
    $d_R(X_{t_\beta},X_s) = O(1)$. On the other hand, if $\beta$ is
    $(n_0,L_0)$--horizontal, then its length grows essentially
    exponentially from its balanced time $t_\beta < s$ onward by
    Proposition \ref{Prop:Growth}. The conclusion now follows from the fact
    that $\ell_s(\beta)$ is bounded. To proceed, we note that
    $d_R(X_s,\beta) = O(1)$, and recall that $\alpha$ is $1$--close to
    $\pi_R(P_t)$ and $\diam_R(\g([t,s]))$ is bounded. Thus, by the triangle
    inequality,
    \[ d_R(\alpha,\beta) \ladd d_R(\alpha,X_t) + d_R(X_t,X_s) +
    d_R(X_s,\beta) = O(1). \] As $\Pi$ is Lipschitz by the first part of
    the Theorem, the 
    inequality above implies that $$d_R(X_{t_\alpha},X_{t_\beta}) = O(1).$$ 
    A final application of the triangle
    inequality now yields 
    \[ d_R(X_{t_\alpha},X_t) \le d_R(X_{t_\alpha},X_{t_\beta}) +
    d_R(X_{t_\beta},X_s) + d_R(X_s,X_t) = O(1).\qedhere\]
    
  \end{proof}
 
  In light of \thmref{Reparam}, we obtain the following corollary.

  \begin{corollary} \label{Cor:Reparam}
    
    Let $\g \colon [a,b] \to \T(\s)$ and let $[c,d] \subseteq [a,b]$ be the
    active interval of a subsurface $R$ along $\g$. Then
    $\Upsilon_R(\g([c,d])$, the shadow of $\g|_{[c,d]}$ to $\calC(R)$,  is
    a reparametrized quasi-geodesic in $\calC(R)$. 

  \end{corollary}

  \begin{remark} \label{BalancedTime}
   
    Let $\g \colon [a,b] \to \T(\s)$ be a Thurston geodesic, and recall
    that for any curve $\alpha \in \calC(R)$, our definition of the
    balanced time $t_\alpha$ is the first time that $\alpha$ is
    $(n_0,L_0)$--horizontal in the active interval $[c,d]$ of $R$ along
    $\g$. However we can also define $t_\alpha'$ to be the first time
    $\alpha$ is $(n_0,L_0)$--horizontal in $[a,b]$ (see also \cite[\S
    5.1]{LRT2}). The two times are equivalent from the point of view of
    coarse geometry of $\calC(R)$. More precisely, we claim
    $d_R(X_{t_\alpha},X_{t_\alpha'})$ is uniformly bounded. To see this,
    note that by definition, if $t_\alpha \ne t_\alpha'$, then either
    $t_\alpha' < c$ or $t_\alpha' > d$. If $t_\alpha' < c$ and $t_\alpha =
    c$, then the conclusion follows from \thmref{AI}\ref{ai1}. If
    $t_\alpha' < c$ and $t_\alpha > c$, then the conclusion follows by
    \thmref{Horizontal}. Finally, if $t_\alpha' > d$, then $t_\alpha = d$
    by definition. In this case, the conclusion again follows by
    \thmref{AI}\ref{ai1}. Hence, from the point of view of the coarse
    geometry of $\calC(R)$, there is not any difference between the two
    definitions of the balanced time.
    
    Finally, note that by \thmref{AI}, $\diam_R(\g([a,c]))$ and
    $\diam_R(\g([d,b]))$ are bounded which imply that
    $\Upsilon_R(\g([a,b]))$ is also a reparametrized quasi-geodesic in
    $\calC(R)$.

  \end{remark}

  \subsection{Nearly filling}

   In this final section after introducing the notion of a {\em nearly
   filling lamination}, we show that when the maximal stretch lamination of
   a Thurston geodesic segment is nearly filling at the initial point of
   the segment, then the shadow of the geodesic is a reparametrized
   quasi-geodesic in the curve complex of any subsurface that intersects
   the stretch lamination.

  \begin{definition}[Nearly filling] \label{Def:Nearly-filling}
   Let $\g \colon I \to \T(\s)$ be a Thurston geodesic, then we say that
    $\lambda_\g$ is \emph{nearly filling} on $X_t=\g(t)$ if the inequality
    $d_\s(X_t,\lambda_\g) \ge 4B$ holds.
   \end{definition}

  \begin{theorem} \label{Thm:Filling}

    Let $\g \from [a,b] \to \T(\s)$ be a Thurston geodesic and suppose that
    $\lambda_\g$ is nearly filling on $X_a$. Then for all $t \ge  a$, any
    curve $\alpha$ disjoint from $\lambda_\g$ is horizontal on $X_t$.
    
  \end{theorem}

  \begin{proof}
    
    Since $\alpha$ is disjoint from $\lambda_\g$, $R=S \setminus \alpha$
    contains $\lambda_\g$. The inclusion map $\iota : \calC(R)\to \calC(S)$
    is a contraction, so $d_R(X_a,\lambda_\g) \ge 4B$. Thus, for all $t \ge
    a$, by Lemma \ref{Lem:AI}, $d_R(X_t,\lambda_\g) > 1$. We assume that
    $\alpha$ is not horizontal on $X_t$ to derive a contradiction. This
    assumption leads to $\ell_t(\alpha) \leq \rho$ by Proposition
    \ref{Prop:Trichotomy}, and also that $d_S(X_t,\alpha) \le B$, by
    definition of $B$ as in Lemma \ref{Lem:ConstantB}.
    
    Consider the active interval for $R$ along $\g$, which must be of the
    form $[a,d]$ for some $d \le b$. We now analyze the two cases that
    either $[a,d]$ is a trivial or non-trivial interval. 

    If $a<d$, then $\ell_a(\alpha) \le \rho$ by Theorem \ref{Thm:AI}. In
    this case, $d_S(X_a,\alpha) \le B$ by the definition of $B$. But then
    \[ d_S(X_a,\lambda_\g) \le d_S(X_a,\alpha) + d_S(\alpha,\lambda_\g) \le
    B+1 < 2B,\] contradicting the assumption.

    If $a=d$, then $\alpha$ is horizontal on $X_a$. Let $s \in [a,b]$ be
    the first time such that $\alpha$ is $(n_0,L_0)$--horizontal on $X_s$;
    and set $s=b$ if no such $s$ exists.  It is not possible for $t >
    s+s_1$, where $s_1 \ge \log \frac{\rho}{w_{\ep_h}}$. This is because
    on  $X_s$, $\alpha$ must intersect some $\ep_h$--short curve, so
    $\ell_s(\alpha) \ge 2w_{\ep_h}$. But then $\ell_t(\alpha) > \rho$ by
    Proposition \ref{Prop:Growth}, which is not possible. On the other
    hand, $\diam_S(\g([a,s])) \le 1$ and $\diam_S(\g([s,s+s_1]) \le B$.
    So if $t \in [a,s+s_1]$, then 
    \begin{align*}
      d_S(X_a,\lambda_\g) 
      & \le d_S(X_a,X_t) + d_S(X_t,\alpha) + d_S(\alpha,\lambda_\g) + 4 \\
      & \le (B+1)+B+1+4 \\
      & =2B+6 < 4B.
    \end{align*}
    This again contradicts the assumption. The proof that $\alpha$ must be
    horizontal on $X_t$ is now complete. \qedhere   
  
  \end{proof}
  
  Now we restate Theorem \ref{Thm:filling} from the introduction and prove 
  in as a corollary of our results.
   
  \begin{corollary}

    Let $\g \colon [a,b] \to \T(\s)$ be a Thurston geodesic. If $\lambda_\g$ is
    nearly filling on $X_a$, then for any subsurface $R \subseteq \s$,
    the set $\Upsilon_R(\g([a,b]))$ is a reparametrized quasi-geodesic in
    $\calC(R)$.
    
  \end{corollary}

  \begin{proof}

    If $\lambda_\g$ intersects $R$, then the
    conclusion follows from \corref{Reparam}. If $\lambda_\g$ does not
    intersect $R$, then it is disjoint from $\bd R$. By \thmref{Filling},
    $\bd R$ is horizontal and the conclusion follows because
    $\diam_R(\g[a,b])$ is bounded by \propref{Horizontal2}. \qedhere

  \end{proof}


  \bibliographystyle{amsalpha}
  \bibliography{main}
  \end{document}